\documentclass[final,reqno,twoside,a4paper,11pt]{amsart}
\usepackage{mltools}
\nodraftlayout
\withdetails
\usepackage{mlthm10-REAR}

\usepackage{upref,amsmath,amssymb,amsthm,mathrsfs}
\usepackage[colorlinks=true,linkcolor=blue,citecolor=blue]{hyperref}

\usepackage[scaled]{helvet} % ss
\usepackage{courier} % tt
\usepackage[mathbf]{euler}
\usepackage{mllocal}

\numberwithin{equation}{section} 
\begin{document}
\title{The $\KO$--valued spectral flow for skew-adjoint Fredholm
operators}

\author{Chris Bourne}
\address{WPI-Advanced Institute for Materials Research (WPI-AIMR), Tohoku University,
2-1-1 Katahira, Aoba-ku, Sendai, 980-8577, Japan and 
RIKEN iTHEMS, Wako, Saitama 351-0198, Japan}
\email{chris.bourne@tohoku.ac.jp}
\thanks{C.B. is supported by a JSPS Grant-in-Aid for Early-Career Scientists (No. 19K14548).}

\author{Alan L. Carey}
\address{Mathematical Sciences Institute, Australian National University, 
Kingsley St., Canberra, ACT 0200, Australia 
and School of Mathematics and Applied Statistics, University of Wollongong, NSW, Australia,  2522}  
\email{acarey@maths.anu.edu.au}
%\email{acarey@maths.anu.edu.au}  
\urladdr{http://maths.anu.edu.au/~acarey/}
\thanks{A.L.C. acknowledges the financial support of the Australian Research Council.}

\author{Matthias Lesch}
\address{Mathematisches Institut,
Universit\"at Bonn,
Endenicher Allee 60,
53115 Bonn,
Germany}
\email{ml@matthiaslesch.de, lesch@math.uni-bonn.de}
\urladdr{www.matthiaslesch.de, www.math.uni-bonn.de/people/lesch}
\thanks{M.L. would like to thank the following institutions:
The School of Mathematics and Applied Statistics of the University of Wollongong
and the University of Sydney for their hospitality, the international
visitor program of the University of Sydney for financial support,
the University of Bonn for granting a sabbatical semester and the
Hausdorff Center for Mathematics, Bonn, for financial support. }

\author{Adam Rennie}
\address{School of Mathematics and Applied Statistics, University of Wollongong, NSW, Australia,  2522}
\email{renniea@uow.edu.au}
\thanks{C.B., A.L.C. and A.R. thank the Erwin Schr\"{o}dinger Institute program 
Bivariant K-Theory in Geometry and Physics for hospitality and support during the 
production of this work.}


\dedicatory{This article is dedicated to Krzysztof Wojciechowski, our
friend and colleague whom we have missed for over a decade.  His
interest and contributions to index theory and geometry have
been a constant source of inspiration.}

\newcommand{\Alan}{\textbf{Alan}}
\newcommand{\Adam}{\textbf{Adam}}
\newcommand{\Chris}{\textbf{Chris}}
\newcommand{\Matthias}{\textbf{Matthias}}
\begin{abstract}
  In this article we give a comprehensive treatment of a `Clifford module
flow' along paths in the skew-adjoint Fredholm operators on a real Hilbert space that takes 
values in $\KO_*(\R)$ via the Clifford index of
Atiyah--Bott--Shapiro.  We develop its properties for both bounded and
unbounded skew-adjoint operators including an axiomatic characterization.
Our constructions and approach are motivated by the principle that 
\[
  \text{spectral flow} = \text{Fredholm index}.
\]
That is, we show how the $\KO$--valued spectral flow relates to a
$\KO$--valued index by proving a  Robbin--Salamon type result.
The Kasparov product is also used to establish a $\text{spectral flow} = \text{Fredholm index}$ 
result at the level of
bivariant $\KU$-theory. We explain how our results incorporate previous applications of
$\Z/ 2\Z$--valued spectral flow in the study of topological phases of matter.
\end{abstract}

\maketitle

\tableofcontents

% vim: nosmartindent
%############################################
\section*{Introduction}

In this paper we examine the theory and applications of spectral flow for operators 
on a real Hilbert space. Aspects of the theory have been in place since Atiyah and 
Singer~\cite{AS69}. Recent developments in the mathematical study of topological phases 
necessitates a more systematic approach.

\subsection*{Background}
\label{s.Alan}
%############################################

The first occurrence of the notion of spectral flow for loops in the
space of skew-adjoint Fredholm operators on a real Hilbert space
appears in an article of Witten \cite{Wit82}.  From a more
mathematical perspective the first study is due to Lott \cite{Lot88}
who, drawing inspiration from Witten's ideas, introduced the notion of
`Clifford module-valued spectral flow' for loops in the classifying
spaces of real $\KU$-theory. He followed the treatment by Atiyah--Singer
for these classifying spaces \cite{AS69}.  These early papers also
closely followed the article of Atiyah, Patodi, and Singer \cite[Sec.
7]{APS:SARIII} by introducing spectral flow using the notion of
`intersection number'. 

The present article is partly motivated by recent applications of the
notion of `real spectral flow' in the study of topological phases of
matter \cite{CPSchuba19,DSBW} where skew-adjoint Fredholm operators
arise naturally from Hamiltonians in quantum systems.  Our focus,
however, is to give a comprehensive treatment of the mathematical
questions that these applications have raised. We nevertheless sketch
how notions of spectral flow that have previously appeared in the
study of topological phases relate to our results.

Our first observation is that the intersection number definition of
spectral flow  is not helpful for the applications to topological
phases that have emerged recently.  For these, an analytic approach is
needed, particularly for aperiodic and disordered systems whose spectra 
cannot be easily studied in general.   For spectral flow
along paths of self-adjoint Fredholm operators on a complex Hilbert
space, an analytic approach began in \cite{Phi96} and was extended in
\cite{BLP05}.  This analytic approach may be adapted to the case of
spectral flow along paths in the skew-adjoint Fredholm operators on a
real Hilbert space. One method of achieving this is found in
\cite{CPSchuba19} where a $\Z/2\Z$--valued spectral flow in the
skew-adjoint Fredholm operators on a real Hilbert space was defined. 

The study of topological phases can be related to 
the classifying spaces of $\KO$-theory~\cite{AZ97, Kitaev09}.
 Here we exploit the construction by Atiyah--Singer
\cite{AS69, Kar70, Sch93} of these spaces in terms of skew-adjoint Fredholm
operators and develop in this framework a general analytic theory of
spectral flow. 

%*********************************************************
\subsection*{The main results}
\label{ss.Alan.MR}
%*********************************************************

Our first objective is the generalization of the analytic
$\Z/ 2\Z$--valued spectral flow of~\cite{CPSchuba19}  along paths in the space of 
skew-adjoint Fredholm operators on a real Hilbert space to a `Clifford
module flow' that takes values in $KO_*(\R)$ via the Clifford index of
Atiyah--Bott--Shapiro~\cite{ABS64}.  After setting out the
preliminaries in the early sections, we give the definition in Sec.~\ref{s.ClkSF} 
and its relation to all previously studied spectral flow. 
We develop its properties for both bounded and
unbounded skew-adjoint operators  and  present an
axiomatic framework for this $\KO$--valued spectral flow that is
analogous to that already existing in the complex case in \cite{L05}.

Our constructions and approach are motivated by the principle that 
\[
  \text{spectral flow} = \text{Fredholm index}.
\]
That is, we show how $\KO$--valued spectral flow relates to a
$\KO$--valued index by proving in Theorem~\ref{thm:Robbin-Salamon} in
Sec.~\ref{s.RS} a Robbin--Salamon type result \cite{RobSal:SFM}.
Kasparov's bivariant theory has also come to play a role in
applications to models of topological phases of matter. With this in
mind we extend Theorem~\ref{thm:Robbin-Salamon} to show how the
Kasparov product may be used to establish a 
$ \text{spectral flow} = \text{Fredholm index} $ result at the level
of bivariant K-theory.  The proof  is a straightforward application of
the constructive Kasparov product \cite{KaaLes:SFU}. 

\subsection*{Organization of the paper}

After a brief introduction to the physical motivations underpinning our 
work in Sec.~\ref{ss.Alan.MPT}, 
Sec.~\ref{s.ABSC} gives a careful introduction to Clifford algebras and
to the ABS Clifford index \cite{ABS64}. 
As it will be
useful for applications we go into some explicit details.  Then 
Sec.~\ref{s.NHCT} expands on \cite{AS69} and contains
preliminaries needed for later proofs.

The essentially new material starts in Sec.~\ref{s.FPCI} where we
introduce a $\KO$--valued index for Fredholm pairs of complex
structures on a real Hilbert space. This is the real counterpart to
the long-standing concept of Fredholm pairs of projections.  There is
of course a link between them which we explain. We note that this
section is directly relevant to the applications to topological
phases.

The $\KO$--valued spectral flow for paths of bounded skew-adjoint
Fredholm operators is introduced next in Sec.~\ref{s.ClkSF}. We
collect  its fundamental properties and show that it incorporates
previous examples of real spectral flow that have been discussed in
the literature. The extension to the unbounded case then follows in
Sec.~\ref{s.EUO}.

Our $\KO$--valued spectral flow can be characterized uniquely in an
axiomatic fashion as we explain in Sec.~\ref{s.USF}. This
characterization enables us to give a short proof of our
Robbin--Salamon type theorem in Sec.~\ref{s.RS}. Finally, in Sec.~\ref{s.SFKP} 
we relate this theorem to Kasparov theory by exploiting
the unbounded Kasparov product. The latter viewpoint is relevant for 
applications to topological phases and the bulk-edge
correspondence. We collect some homotopy results from \cite{AS69, Mil63} that will be of 
use to us in Appendix~\ref{s.SCHE}.

%*********************************************************
\section{Motivation from physical theory}
\label{ss.Alan.MPT}
%*********************************************************

As we remarked earlier, motivation for this paper comes from issues
raised by the mathematical aspects of studies of topological phases
for fermionic one particle systems.  Historically model Hamiltonians of these
fermion systems commute or anti-commute with prescribed symmetry
operators that may arise as unitary or anti-unitary involutive operators on a complex
Hilbert space. Such anti-unitary symmetries then determine a real
structure on the complex space and one may consider topological phase
labels of Hamiltonians that respect given symmetries and real
structure. Our results in this paper show that these topological phase
labels are identified by $\KO$--valued spectral flow.

Previous work, notably~\cite{Kitaev09, FM13} and further developed in 
(amongst other places) \cite{Thiang16} showed that the symmetries of the Hamiltonians
that allow topological phases generate real Clifford algebras. This leads on to the idea of
detecting different topological phases using  the Clifford module valued index  of
Atiyah--Bott--Shapiro.  Some of the present authors showed that one could apply
Kasparov's bivariant theory in the real case to realize these Clifford
module indices~\cite{BCR2,BKR17}.

%{\color{blue} 
The connection of topological phases with real index theory of 
skew-adjoint Fredholm operators arises in its simplest form as follows.
There is a class of model Hamiltonians that are said to be of 
Bogoliubov--de Gennes type (these are called symmetry class D in the 
physics literature). That is, there is a complex 
Hilbert space $H_\C$ equipped with a particle-hole symmetry 
given by an anti-linear  involution $C$ on $H_\C$.
The symmetry operator $C$ is self-adjoint, $C^2=1$ and  
$C\mathbf{H}C = -\mathbf{H}$ with $\mathbf{H}=\mathbf{H}^*$ 
the (bounded) self-adjoint Hamiltonian of the system. 
Then the operator $i\mathbf{H}$ commutes with $C$ and gives a skew-adjoint 
operator on the real Hilbert space 
$H_\R = \{v \in H_\C \,:\, C v = v\}$. 
It is normally assumed in these models that 
$0 \notin \sigma(\mathbf{H})$ and so $i\mathbf{H}$ is an invertible (and therefore Fredholm) 
skew-adjoint operator on a real Hilbert space. 

The space of all skew-adjoint Fredholm operators commuting with $C$
has non-trivial topology.  One method of detecting this topology is
to consider paths of such Hamiltonians $\{\mathbf{H}_t\}_{t\in [0,1]}$ that respect the 
particle hole symmetry $C \mathbf{H}_t C = -\mathbf{H}_t$ for all $t \in [0,1]$. If we restrict to 
Fredholm paths, where $\ker(\mathbf{H}_t)$ is finite dimensional for all $t$,  we obtain a 
path of skew-adjoint Fredholm operators $\{i \mathbf{H}_t\}_{t\in [0,1]}$ on the real Hilbert space $H_\R$. 
If we also assume the endpoints 
$\mathbf{H}_0$ and $\mathbf{H}_1$ are invertible (e.g. we consider Fredholm paths connecting 
two gapped Bogoliubov--de Gennes Hamiltonians), this path can be completed into a loop. 
As such, topological properties of Hamiltonians in symmetry class D are closely connected to the 
loop space of skew-adjoint Fredholm operators, as was studied in detail in~\cite{CPSchuba19}.
Such loops give us the simplest example of $\KO$--valued spectral flow and 
was termed $\Z/2\Z$ spectral flow.
%}

It was observed in~\cite{BSB19} that the $\Z/2\Z$--valued spectral
flow provides an alternative method to describe the relative 
topological phase of Bogoliubov--de Gennes Hamiltonians  without 
any other information (e.g. dimension or a Brillouin torus).  
 Because the
topological phase of gapped free-fermion Hamiltonians  may take values
in any of the groups $\KO_{\ast}(\R)$, there is a clear motivation to
fully develop the notion of a general $\KO$--valued spectral flow that
incorporates systems with symmetries. A related point of view may also
be found in~\cite{DSBW}.

%{\color{blue} 
%Much like our mathematical aim of the manuscript is a generalization of the 
%$\Z / 2\Z$-valued spectral flow~\cite{CPSchuba19}, our motivation from physics is 
%to extend the simple connection of spectral flow with  Hamiltonians in symmetry class D 
%to all free-fermionic symmetry classes. It is on this point that
To consider paths of Hamiltonians with symmetries and a 
generalized spectral flow, the picture of free-fermionic topological 
phases developed by Kennedy and Zirnbauer~\cite{KZ16} and more recently 
Alldridge, Max and Zirnbauer~\cite{AMZ19} is of particular use to us. 
Following this description, 
we again consider  a self-adjoint 
Bogoliubov--de Gennes Hamiltonian $\mathbf{H}$, $C\mathbf{H}C = -\mathbf{H}$, 
such that  
$\mathbf{H}$ commutes with prescribed physical symmetry 
operators~\cite{AZ97, HHZ05}. 
A key observation described in detail in~\cite{KZ16} is that 
these physical symmetry operators can be used to construct a representation 
of a real Clifford algebra acting on the real Hilbert space $H_\R$ of 
elements fixed by $C$. Furthermore, the 
generators of this Clifford representation anti-commute with the 
real skew-adjoint operator $i\mathbf{H}$ and its `spectral flattening' 
$J_{\mathbf{H}} = i\mathbf{H}|\mathbf{H}|^{-1}$. That is, free-fermionic gapped Hamiltonians 
with Altland--Zirnbauer-type symmetries are in one-to-one correspondence 
with skew-adjoint gapped operators on a real Hilbert space that anti-commute 
with the generators of a Clifford algebra representation. Representations of 
all ten Morita classes of 
Clifford algebras can be constructed by considering the different Altland--Zirnbauer symmetry types. 
Topological phases have been associated to such systems by classical
homotopy methods or van Daele $\KU$--theory~\cite{KZ16, AMZ19}. 
See also~\cite{Kel17, K19}.
%}

%{\color{blue} 
Given a fixed Altland--Zirnbauer symmetry type, we can 
consider paths of Hamiltonians $\{\mathbf{H}_t\}_{t\in[0,1]}$ that respect  
the given symmetry for all $t \in [0,1]$. 
Passing to the real subspace fixed under the anti-linear involution, 
the path $\{i\mathbf{H}_t\}_{t\in[0,1]}$ will anti-commute with the Clifford generators 
that come from the symmetry type for all $t$. By restricting to 
Fredholm paths with invertible endpoints, we therefore obtain a 
loop of skew-adjoint Fredholm operators on a real Hilbert space that anti-commute 
with a fixed number of Clifford generators.
By constructing a $\KO$--valued spectral flow that defines a homomorphism of  the loop 
group of such skew-adjoint Fredholm operators onto an Atiyah--Bott--Shapiro index group, 
we obtain a mathematically 
precise formulation of the physical description of topological phases as homotopy 
classes of symmetric Hamiltonians. That is, the 
$\KO$--valued spectral flow precisely encodes the topological
obstruction for two symmetric and gapped free-fermion Hamiltonian
systems to be connected by a continuous Fredholm path that respects the
underlying symmetries. Conversely, a non-trivial $\KO$--valued spectral
flow guarantees the existence of topological zero energy states (i.e.
a kernel) at some point along the path joining two gapped Hamiltonians.
%}

%{\color{blue}
Let us also note that the application of spectral flow to systems and models in 
condensed matter physics has a rich history. The connection between charge transport, 
Chern numbers and the Hall conductance with the index of a pair of projections 
(used heavily in Phillip's definition \cite{Phi96} of spectral flow) was 
studied in detail by Avron, Seiler and Simon~\cite{ASS94, ASS94b}. 
The index of a pair of projections continues to provide a useful mechanism 
to characterize bulk topological phases in the physics literature, see~\cite{LiMong} for example.
In the mathematical physics literature, the index of skew-adjoint 
Fredholm operators and generalizations of the index of a pair of projections 
have been used to characterize the strong topological phase in all dimensions 
and symmetry classes~\cite{GSB, KatsuraKoma}.

Recently 
Schulz-Baldes and co-authors have used spectral flow constructions to study 
the strong invariant of topological insulators (potentially with 
anti-linear symmetries) in two-dimensional systems~\cite{DSB} and 
in complex classes of topological phases (symmetry type A and AIII) 
in arbitrary dimension~\cite{CSB19}. 
Such spectral flows are usually implemented physically via a flux or monopole 
insertion through a local cell in a lattice system.
 Our work more closely 
follows~\cite{CPSchuba19,DSBW}, which concerns the homotopy type of 
Hamiltonians and Fredholm operators with certain symmetries and without reference 
to a Brillouin zone or pairing with a Dirac-type element in $\KU$--homology, 
some basic examples are given in Sec.~\ref{ss.TopPhaseExamples}.
Therefore, while our $\KO$--valued spectral flow covers all symmetry classes of 
free-fermionic topological phases, it is in general an additional step to relate this 
spectral flow to a strong or weak topological phase.

While our results in the current paper shed new light on $\KO_r$--valued spectral flow for all
$r$, the instance that is completely unexplored previously is that of $\KO_4$. 
Because our constructions do not make reference to lattice dimension, $KO_4$--valued 
spectral flow corresponds to Hamiltonians in symmetry class AII.
In this case the kernels of the skew-adjoint
Fredholm operators are quaternionic Hilbert spaces. Thus we obtain from our results the notion of quaternionic 
spectral flow.  One may write down explicit Hamiltonians which realize this case though the construction is lengthy
and we will not attempt it here.  These model Hamiltonians are however of the general form
that we describe later in Sec.~\ref{ss.TopPhaseExamples}.
%}

%A complementary approach is given by Kennedy and Zirnbauer~\cite{KZ16}
%and more recently Alldridge, Max and Zirnbauer~\cite{AMZ19}, who only
%consider symmetries that commute with a fermionic Hamiltonian.  These
%Hamiltonians are termed `gapped' in that we assume a spectral gap
%containing zero.  It is then shown that for the symmetry types of
%interest in such fermion systems, \cite{AZ97, HHZ05}, there is a
%one-to-one correspondence between gapped Hamiltonians $\mathbf{H}$ commuting
%with symmetry operators and `pseudo-symmetry' operators that
%anti-commute with $J_\mathbf{H} = i\mathbf{H}|\mathbf{H}|^{-1}$, 
%a complex structure on a real
%Hilbert space.  
%Topological phases are then associated to such systems by classical
%homotopy methods or van Daele $\KU$-theory~\cite{KZ16, AMZ19}. 
%See also~\cite{Kel17, K19}.

%We will not emphasize the physical application of our construction,
%but note that our index on Fredholm pairs of complex structures
%exactly covers the case of pairs $(J_{\mathbf{H}_1},  J_{\mathbf{H}_2})$ anti-commuting
%with pseudo-symmetries developed in~\cite{KZ16, AMZ19}.  That is, the
%$\KO$--valued spectral flow precisely encodes the topological
%obstruction for two symmetric and gapped free-fermion Hamiltonian
%systems to be connected by a continuous path that respects the
%underlying symmetries. Conversely, a non-trivial $\KO$--valued spectral
%flow guarantees the existence of topological zero energy states (i.e.
%a kernel) at some point along the path joining two symmetric free
%fermion Hamiltonians.

%############################################
\section{Clifford algebras and the ABS construction}
\label{s.ABSC}

Since it will be important for us we briefly recall here the ABS
construction, cf. \cite{ABS64,LM89}. We put an emphasis, though,
on the real Clifford algebras $\Clrs$ with $r$ self-adjoint and
$s$ skew-adjoint generators. The extension of the ABS construction
to this case, cf. \cite{At66, Kar70}, is straightforward, but as
far as we know, not easily accessible in published form.

%*********************************************************
\subsection{Preliminaries and notation}
\label{ss.PN}
%*********************************************************

%{\color{blue} 
The purpose of this section is to fix some terminology and notation. It
will be in effect through
the whole paper. This section should be consulted first in case the
reader is looking for the definition
of a certain notation.
%}

By $\Z, \R, \C$ we denote the sets of integers, real numbers, and complex
numbers resp.

Let $\Clrs$ denote the (real) Clifford algebra on $r+s$ 
unitary generators
$e_1,\ldots,e_r$, $f_1,\ldots , f_s$ subject to the relations 
\begin{equation}\label{eq.Cliff-Relations}
  \begin{split}
      & e_i e_j + e_j e_i = 2\delta_{ij}, \quad f_k f_l + f_k f_l =
      -2\delta_{kl}, 
           \quad e_i f_k + f_k e_i = 0,\\% \quad 1\le k\le r, 1\le l\le s.
      &f_j^* =-f_j,\quad e_j^*=e_j, \quad \text{for } 1\le i,j\le r,\ \  1\le k,l\le s.
  \end{split}
\end{equation}
The algebra $\Clrs$ is a $\Z/2\Z$--graded algebra by declaring the generators to be of odd
parity. The product of all generators
\begin{equation}
      \go_{r,s}:= e_1\cldots e_r\cdot f_1\cldots f_s,
      \quad \go^2_{r,s} = 
      \begin{cases} (-1)^r, & r+s\equiv 3 \text{ or } 4\mlmod 4,\\
                   (-1)^{r+1}, & r+s\equiv 1 \text{ or } 2\mlmod 4,
      \end{cases}  
\end{equation}      
is called the \emph{volume element}. It is central if $r+s$ is odd
and it implements the grading if $r+s$ is even.

%{\color{blue}
\begin{dfn}[Clifford element, Clifford generator]
Elements $e_1,\ldots,e_r$, $f_1,\ldots , f_s$ in an algebra (e.g. on a
Hilbert space) satisfying the relations \eqref{eq.Cliff-Relations} will be
addressed as Clifford elements or, if the generated algebra is
emphasized, Clifford generators.
\end{dfn}
%}

Our convention for labeling $\Clrs$ is the same as in
\cite{Kas80,AMZ19} and opposite to \cite{LM89}.  Occasionally, for
illustration purposes we translate results to the complex case, where 
$\CClrs=\CCl_{0,r+s}=\CCl_{r+s,0}=:\CCl_{r+s}$ denotes the complex Clifford
algebra on $r+s$ generators (same relations).

\begin{dfn}[$\Clrs$--Hilbert space, Clifford module]
A real resp. complex Hilbert space is a vector space $H$ over the $\R$ resp. $\C$ together
with an inner product $\inn{\cdot,\cdot}$ such that $H$ together with
the norm $\|x\|:=\sqrt{\inn{x,x}}$ is complete.

Let $H$ be a separable real Hilbert space and let $\sB(H)$ be the
$C^*$--algebra of bounded linear operators on $H$.
We call $H$ a \emph{$\Clrs$--Hilbert space} if $\Clrs$ is represented 
as a $C^*$--algebra on $H$. In this case, the generators are then denoted by 
$E_1,\ldots,E_r$, $F_1,\ldots,F_s\in \sB(H)$.

A \emph{finite dimensional} $\Clrs$--Hilbert space will also be called a
Clifford module, still with the understanding that it is equipped with an
inner product such that the $\Clrs$--action is a $C^*$--algebra
representation.
\end{dfn}

For a $\Clrs$--Hilbert space we denote by 
\begin{equation}\label{eq.NHCT.1}
  \sB^{r,s+1}(H):= \biggsetdef{T\in\sB(H)}{%
    T=-T^*, \; \begin{matrix} E_jT=-TE_j, \; j=\rdots,\\
                              F_kT=-TF_k, \; k=\sdots
                \end{matrix}   }
\end{equation}
the skew-adjoint operators which are $\Clrs$--anti-linear. In
particular, $\sB^{0,1}(H)$ denotes the skew-adjoint bounded operators.

Occasionally, the corner case $r=s=0$ comes up. A $\Cl_{0,0}$--Hilbert
space is just a real Hilbert space with the understanding that
$E_0=I$. Furthermore, we put $\sB^{0,0}(H)=\sB(H)$ (this corresponds
to $s=-1$).

For real and complex Hilbert spaces a bounded linear
operator $u$ is called `unitary' if $u^*u=uu^*=I$, that
is in the real case we use `orthogonal' and `unitary' as synonyms.

A $\Clrs$--Hilbert space will be called \emph{standard} if
the multiplicity of each irreducible representation of $\Clrs$
on $H$ is infinite. 

For separable Hilbert spaces the $\infty$ multiplicity assumption is
not really serious. Firstly, if $r-s\not\equiv 1 \mlmod 4$ it is a non-issue 
as $\Clrs$ has only $1$ irreducible representation 
(\cite[Theorem 5.7]{LM89}, beware the different convention).  If
$r-s\equiv 1 \mlmod{4}$ then there are $2$ irreducible
representations.  Given a concrete $H$ with a $\Clrs$ action we can
always find another separable Hilbert space $H'$ such that $H\oplus
H'$ satisfies the $\infty$ multiplicity assumption and hence is a
standard $\Clrs$--Hilbert space.

We add the remark that if the $\Clrs$--Hilbert space is standard
then it can be written as $H=H'\oplus H'$, where the decomposition
is $\Clrs$--linear and $H'$ itself is standard. Putting $E_{r+1}:=
\begin{pmatrix} 0 & I \\ I & 0\end{pmatrix}$,
$F_{s+1}:= \begin{pmatrix} 0 & -I \\ I & 0\end{pmatrix}$,
$H$ is equipped with a standard $\Cl_{r+1,s+1}$ structure.
Inductively, this can be continued and hence on a standard
$\Clrs$--Hilbert space we may always assume that the infinite Clifford
algebra $\bigcup_{r,s} \Cl_{r,s}$ is acting and that for each $r,s$
the action of $\Clrs$ is standard.  In particular, we may always
assume that there are as many Clifford generators $E_j,F_k$ at our
disposal as we please.

All standard $\Clrs$--Hilbert spaces are equivalent as $\Clrs$--modules
and there is the following obvious \emph{Absorption Theorem}:
given a $\Clrs$--Hilbert space $H$ and a standard $\Clrs$--Hilbert space
$H'$ then $H\oplus H'$ is a standard $\Clrs$--Hilbert space and hence
equivalent to $H'$ as a $\Clrs$--module.

Finally, $\sK(H)$ denotes the $C^*$--algebra of compact operators
and by $\pi: \sB(H) \to \sQ(H)=\sB(H)/\sK(H)$ we denote the quotient
map onto the \emph{Calkin algebra}. For an operator $T\in\sB(H)$
we will use the abbreviation $\|T\|_{\sQ(H)}:=\|\pi(T)\|_{\sQ(H)}$
for the norm of $\pi(T)=T+\sK(H)$ in the Calkin algebra. If 
there is only one Hilbert space around, the disambiguator ``$(H)$'' will
occasionally be omitted.

We will also have to deal with finite dimensional Clifford modules,
where homotopy issues are a priori a bit more delicate. We will,
however always be able to specialize to a \emph{stable} situation,
see Appendix~\ref{ss.SHFC}.

%*********************************************************************************
\subsection{Isomorphisms of Clifford algebras}
\label{ss.ICA}
%*********************************************************************************

For later reference, we introduce the $2\times 2$ matrices
\begin{align}
  K_1 &:= \twomat 1 0 0 {-1}, \quad
  K_2 :=  \twomat 0 1 1 0, \quad
  L_1 :=  \twomat 0 {-1} 1  0,\label{eq.ICA.1}\\
    \go_{1,1}& = K_1\cdot L_1 = -K_2,\quad \go_{2,0}= K_1\cdot K_2 =
    -L_1, \quad K_1\cdot K_2\cdot L_1 = I_2.\label{eq.ICA.2}
\end{align}

Note that, see e.g. \cite[Sec.~I.4]{LM89},
$\Cl_{1,0} = C^*(K_1)\simeq \R\oplus \R$, $\Cl_{2,0} =
C^*(K_1,K_2) = C^*(K_1,L_1) = \Cl_{1,1} = \Mat(2,\R)$, \footnote{ %FOOTNOTE
  For a ring $R$ we denote by $\Mat(k,R)$ the $k\times k$--matrices
  over $R$.}    % END FOOTNOTE
and $\Cl_{0,2}\simeq C^*(K_1\otimes L_1, K_2\otimes L_1)\simeq \HH$,
the latter being the quaternions.

From now on we will  view $\R^2$ as a $\Cl_{2,1}$--module
with generators $K_1, K_2, L_1$, as a $\Cl_{1,1}$--module with
generators $K_1, L_1$ as well as a $\Cl_{0,1}$ resp. $\Cl_{1,0}$--module
with generator $L_1$ resp. $K_1$. 

We note further that $C^*(K_1,K_2, L_1) = \Mat(2,\R)$ is one of
the two irreducible representations of $\Cl_{2,1}\simeq \Mat(2,\R)\oplus
\Mat(2,\R)$. The representation is characterized by the fact that
the (representation of the) \emph{volume element} 
$\go_{2,1} := K_1\cdot K_2\cdot L_1 = I_2$ is the $2\times 2$ unit
matrix.

Our choice of matrices $K_1$ and $K_2$ is of course arbitrary and 
other choices are possible. For example, we can exchange $K_1$ 
and $K_2$ by conjugating by the unitary $\frac{1}{\sqrt{2}}(K_1+K_2)$. Though 
we note that this transformation will reverse the orientation, where we 
recall $\omega_{2,0} = K_1K_2$ and similarly will send $L_1 \mapsto -L_1$. 
Hence such transformations can introduce a global minus sign 
and should be used with care.

\begin{prop} \label{p.ICA.1}
\textup{ 1.} Recall \cite[Sec.~I.4]{LM89} the following isomorphisms of 
Clifford algebras:
\begin{align}
  \Cl_{r+1,s+1} &\simeq \Clrs\otimes \Cl_{1,1}\simeq M(2,\Clrs)
     \simeq 
     \begin{cases}
                 M(2^{r+1}, \Cl_{0,s-r}), & \text{ if } s\ge r,\\
                 M(2^{s+1}, \Cl_{r-s,0}), & \text{ if } r\ge s,
     \end{cases}  \label{eq.ABSC.1}\\
  \Cl_{r+2,s}   &\simeq \Cl_{s,r}\otimes \Cl_{2,0},     \label{eq.ABSC.2}\\
  \Cl_{r,s+2}   &\simeq \Cl_{s,r}\otimes \Cl_{0,2},     \label{eq.ABSC.3}\\
  \Cl_{r+1,s}   &\simeq \Cl_{s+1,r}.  \label{eq.ABSC.4}
\end{align}  

\textup{2.a} A $\Cl_{r+1,s+1}$--Hilbert space is
$\Cl_{r+1,s+1}$--covariantly isomorphic to $H = H'\otimes \R^2$ with
$H'$ being a $\Cl_{r, s}$--Hilbert space and such that
$E_{r+1}=I_{H'}\otimes K_1$ and $F_{s+1}=I_{H'}\otimes L_1$.

\textup{2.b} Similarly, a $\Cl_{r+2,s}$--Hilbert space is
$\Cl_{r+2,s}$--covariantly isomorphic to 
$H = H'\otimes \R^2$ with $H'$ being a $\Cl_{s, r}$--Hilbert
space and such that $E_{r+j}=I_{H'}\otimes K_j$, $j=1,2$.

\textup{2.c} Finally, a \emph{standard} $\Cl_{1,0}$--Hilbert space is covariantly
isomorphic to $H'\otimes \R^2$ with $E_1=I_{H'}\otimes K_1$.
\end{prop}
\newcommand{\cliffisosthree}{\Eqref{eq.ABSC.1}--\eqref{eq.ABSC.3}}
\newcommand{\cliffisos}{\Eqref{eq.ABSC.1}--\eqref{eq.ABSC.4}}

The decompositions in 2.a--2.c will be used frequently below.
The tensor products in \Eqref{eq.ABSC.1}--\eqref{eq.ABSC.3} are
ordinary tensor products of graded algebras\footnote{% FOOTNOTE
This statement might be confusing: if $A$ and $B$ are $\Z/2\Z$--graded algebras 
their tensor product $A\otimes B$ is graded by putting 
$\deg(a\otimes b):= \deg a + \deg b \mlmod 2$ for homogeneous elements $a\in A, b\in B$.
This is to be distinguished from the graded (skew-commutative) tensor product
$A\hat\otimes B$. The algebra $A\hat \otimes B$ has the same
underlying vector space $A\otimes B$ and the same grading, but the
product is $a\hat\otimes b \cdot a'\hat\otimes b' 
   := (-1)^{\deg b \deg a'} aa'\hat\otimes bb'$.
See also \cite[Sec.~2.6]{Kas80}.
} % END FOOTNOTE  
and the isomorphisms respect the natural grading.
For the next section it will be crucial that the isomorphism
\Eqref{eq.ABSC.4} does \emph{not} respect the grading. 

\begin{proof}
Let us first comment on the decompositions 2.a--2.c. In 2.c
the standard assumption guarantees that the eigenspaces
of $E_1$ are infinite--dimensional and hence isomorphic.

{2.a/2.b}
The first two cases may be treated in parallel by putting
$G=F_{s+1}$ in the first case and $G:= E_{r+2}$ in the second. 
Since $E_{r+1}$ and $G$ anti-commute, $G$
interchanges the two $\pm 1$--eigenspaces of $E_{r+1}$.
Let $H':= \ker (E_{r+1}-I)$ be the $+1$ eigenspace of $E_{r+1}$
and put
\begin{equation}
    \Psi: H'\otimes \R^2 \longrightarrow H,\quad 
      \vectwo x y \mapsto x+ Gy.
\end{equation}
Then one checks that 
\begin{equation}
    \Psi\ii E_{r+1} \Psi = I_{H'}\otimes K_1,
    \quad \Psi\ii G \Psi = \begin{cases} 
            I_{H'}\otimes K_2, & G=E_{r+2},\\
            I_{H'}\otimes L_1, & G=F_{s+2},
                             \end{cases} 
\end{equation}
and for any other Clifford element $X\in \{\Edots, \Fdots\}$
\begin{equation}
    \Psi\ii X \Psi = \begin{cases} 
      X E_{r+2} \otimes \omega_{2,0}, & G=E_{r+2}, \quad \omega_{2,0}= K_1\cdot K_2 = -L_1,\\
      F_{s+1} X \otimes \omega_{1,1}, & G=F_{s+1}, \quad 
              \omega_{1,1} = K_1\cdot L_1 = - K_2.
                     \end{cases} 
\end{equation}
Thus, letting $\Cl_{r,s}$ (resp. $\Cl_{s,r}$) act on $H'$ with generators
$F_{s+1} X$ (resp. $X E_{r+2} $), $X\in \{ \Edots,$ $ \Fdots \}$, 
$\Psi\ii$ is the isomorphism we are looking for. Note, that
$(X E_{r+2})^2 = - X^2$ and $(X F_{s+1})^2= X^2$, thus
$X E_{r+2} $, $X\in \{ \Edots$, $\Fdots \}$ satisfy the
$\Cl_{s,r}$--relations.

{1.} The previous consideration immediately leads to the proof of 
\Eqref{eq.ABSC.1} and \Eqref{eq.ABSC.2}. Namely, the isomorphism
$\Cl_{r+1,s+1}\to\Clrs\otimes \Cl_{1,1}$
is given by sending $\{ e_1,\ldots, e_r, f_1, \ldots, f_s \}\ni X\mapsto X \otimes
\go_{1,1}$ and $f_{s+1}\mapsto 1\otimes L_1$, $e_{r+1}\mapsto
  1\otimes K_1$. Note that $\go_{1,1}^2=I_2$.

Similarly, the isomorphism $\Cl_{r+2,s}\to \Cl_{s,r}\otimes
\Cl_{2,0}$ is given by sending 
$e_j \mapsto f_j' \otimes \go_{2,0}$, $j=1,\ldots,r$,
$f_k\mapsto  e_k'\otimes \go_{2,0}$, $k=1,\ldots, s$,
  and $e_{r+2}\mapsto 1\otimes K_2$, $e_{r+1}\mapsto 1\otimes K_1$.
Here, $e_k', f_j'$ denote the generators of $\Cl_{s,r}$. 
Note that $\go_{2,0}^2=-I_2$.

Analogously, the isomorphism $\Cl_{r,s+2}\to \Cl_{s,r}\otimes
\Cl_{0,2}$ is given by sending 
$e_j \mapsto f_j' \otimes \go_{0,2}$, $j=1,\ldots,r$,
$f_k\mapsto  e_k'\otimes \go_{0,2}$, $k=1,\ldots, s$,
and $f_{s+1}\mapsto 1\otimes L_1$,
$f_{s+2}\mapsto 1\otimes L_2$, where $L_1, L_2$ are the generators
of $\Cl_{0,2}$. Note that $\go_{0,2}^2=-I_2$.
  
Finally, the isomorphism $\Cl_{r+1,s}\to \Cl_{s+1,r}$ is given
as follows: inside the algebra $\Cl_{r+1,s}$ consider the elements
$e_1':= f_1\cdot e_{r+1},\ldots, e_s':=f_s\cdot e_{r+1}$, $e_{s+1}':=e_{r+1}$,
$f_1':=e_1\cdot e_{r+1},\ldots, f_r':=e_r\cdot e_{r+1}$.
$e_1',\dots, f_r'$ satisfy the Clifford relations for $\Cl_{s+1,r}$.
Hence the map $\Cl_{s+1,r}\to\Cl_{r+1,s}$, $e_j\mapsto e_j',
f_k\mapsto f_k'$ is the (inverse of) the claimed isomorphism.
This also shows that the isomorphism does not respect the grading.
\end{proof}

\begin{remark}\label{p.ICA.2}
{1. } We note in passing a consequence of the
isomorphism \Eqref{eq.ABSC.4}. Given a $\Cl_{r+1,s}$--Hilbert
space with generators $E_1,\ldots, E_{r+1}, F_1,\ldots, F_s$,
it becomes a $\Cl_{s+1,r}$--Hilbert space with generators
$E_j'=F_j\cdot E_{r+1}$, $j=1,\ldots, s$, $E_{s+1}'=E_{r+1}$,
$F_k'=E_k\cdot E_{r+1}$, $k=1,\ldots,r$. Furthermore,
given a $\Cl_{r+1,s}$--anti-linear operator $T\in\sB(H)$
with $T^*=\eps T$, $\eps\in\{\pm 1\}$, then $\tilde T=T E_{r+1}$
is $\Cl_{s+1,r}$--anti-linear with respect to the generators
$E_1',\ldots, F_r'$ and $\tilde T^* = -\eps \tilde T$.
Clearly, if $T$ is Fredholm then so is $\tilde T$.

Thus, there is a parallel theory of $\Clrs$--anti-linear
\emph{self-adjoint} operators. In light of the isomorphism
\Eqref{eq.ABSC.4} the translation is straightforward.
  See, however, the Remark~\ref{p.ClrsFred.3} below.

\subsubsection*{2. } For the decompositions {2.a--2.c} we note
that for $T\in\sB(H)$  with
the notation introduced at the beginning of Sec.~\ref{ss.ICA},
  \begin{equation}\label{eq.ABSC.13}
  \begin{split}
    T_{1,1}&:= \twomat 0 {-T^*} T 0 = \Re(T)\otimes L_1 + \Im(T)\otimes
  K_2,\\ 
     \Re(T)&:= \frac 12 (T+T^*), \quad \Im(T):= \frac 12 (T-T^*).
  \end{split}
\end{equation}  
  The operator $T_{1,1}$ is skew-adjoint and anti-commutes with $K_1$. If
$T$ is already skew-adjoint,  then $\Re(T)=0$ and $T_{1,1}$ anti-commutes with $L_1$
as well. Furthermore, if $T$ already has $\Clrs$--Clifford symmetries then
under the isomorphism \Eqref{eq.ABSC.1} the operator $T_{1,1}$ has
$\Cl_{r+1,s+1}$--Clifford symmetries.
\end{remark}

%*********************************************************************************
\subsection{The ABS construction}
\label{ss.ABS}
%*********************************************************************************

Denote by $\hat\sM_{r,s}$ resp. $\sM_{r,s}$ the Grothendieck group of
graded resp. ungraded finite dimensional $\Clrs$--modules (with
inner product) with addition being direct sum. 
For $s\ge 1$ we have $\hat\sM_{r,s}\simeq \sM_{r,s-1}$ by sending
the graded $\Clrs$--module $V$ to the ungraded $\Cl_{r,s-1}\simeq
\Clrs^0$ module $V^0$.\footnote{The superscript $^0$ refers to
the even subspace w.r.t. the grading.} Here, $\Clrs^0$ is 
identified with $\Cl_{r,s-1}$ by sending $e_i\cdot f_s$ to $e_i$,
$i=1,\ldots, r$ and $f_i\cdot f_s$ to $f_i$, $i=1,\ldots,s-1$.

Also, cf. the second part of Prop.~\ref{p.ICA.1} above, given a graded $\Clrs$--module
$V$ with grading operator $\eps$, $V$ becomes an ungraded
$\Cl_{r+1,s}$--module by sending $e_{r+1}$ to $\eps$. Thus
$\hat\sM_{r,s}\simeq \sM_{r+1,s}$ for all $s$. Altogether, we have
$\sM_{r,s-1}\simeq \sM_{r+1,s}$ and the isomorphism is canonical by
sending the $\Cl_{r,s-1}$--module $V$ to $V\otimes \R^2$ with
$\Cl_{r+1,s+1}\simeq \Clrs\otimes \Cl_{1,1}$ action given by
$X\otimes \go_{1,1}$, $I_V\otimes K_1, I_V\otimes L_1$.

Furthermore, let 
\begin{align}
    A_{r,s}  &:= \hat\sM_{r,s}/\hat\sM_{r,s+1} \,\bl \simeq
        \sM_{r,s-1}/\sM_{r,s}\quad \text{if } s\ge 1\br\label{eq.ABSC.5}, \\
    B_{r,s}  &:= \hat\sM_{r,s}/\hat\sM_{r+1,s} \,\bl \simeq
        \sM_{r,s-1}/\sM_{r+1,s-1}
          \quad \text{if } s\ge 1\br\label{eq.ABSC.6}.
\end{align}  

One of the main results of \cite{ABS64}, \cite[Sec.~I.9]{LM89} 
is that $A_{0,k}$ is canonically isomorphic to $\KO_k(\R)$.

The isomorphisms \Eqref{eq.ABSC.1}--\Eqref{eq.ABSC.4} and
the $8$--periodicity for the Clifford algebras immediately imply
that $A_{r,s}, B_{r,s}$ depend only on the difference $s-r\mlmod 8$.
Furthermore, the ungraded isomorphism \Eqref{eq.ABSC.4}
shows for all $r,\, s\ge 1$
\begin{equation}
  B_{r,s} \simeq \sM_{r,s-1}/\sM_{r+1,s-1} \simeq
  \sM_{s,r-1}/\sM_{s,r} \simeq A_{s,r}.
\end{equation}  
A posteriori, by periodicity, this isomorphism holds for all $r,s$.

The representation theory of the Clifford algebras easily gives
the following table directly:

\medskip
\begin{center}
    \begin{tabular}{|c|c|c|c|c|c|c|c|c|c|}
	\hline
	$s-r\mlmod 8$   &  $0$ & $1$  & $2$   &$3$    &$4$    &$5$   &$6$   &$7$  & $8$\\
  $A_{r,s}\simeq B_{s,r}$  &  $\Z$  & $\Z/2\Z$  & $\Z/2\Z$ & $0$ & $\Z$ & $0$ & $0$ & $0$ & $\Z$ \\
	\hline
\end{tabular}    
\end{center}
\medskip

\begin{dfn}[ABS map]
Below we will explicitly describe \emph{natural} isomorphisms
\begin{equation}\label{eq.ABSC.16}
  \tau_{r,s} : A_{r,s} \longrightarrow 
         % \begin{cases} \Z &, s-r\equiv 0\mlmod 4,\\
         %    \Z/2\Z &, s-r\equiv 1 \text{ or } 2\mlmod 8.
         %  \end{cases}
         \casetwo \Z {s-r\equiv 0\mlmod 4}{\Z/2\Z}{s-r\equiv 1 \text{
           or } 2\mlmod 8.}
         %  \end{cases}
\end{equation}
The map $\tau_{r,s}$ will be called the ABS map.
\end{dfn}

The isomorphisms of the table are unique only up to sign (relevant in
the cases $\equiv 0 \mlmod 4$), hence the choice of Clifford
generators matters in concrete cases. We therefore quickly list here
the concrete maps.

%*********************************************************************************
\subsection{Examples of the \textup{ABS} map for certain $r,s$ }
\label{ss.ABS-examples}
%*********************************************************************************

\subsubsection{ $r=s=0$ } \label{sss.rszero}
This case can only be dealt with in the
graded setting: if $V$ is a graded real vector space then 
$\tau_{0,0}: A_{0,0}\ni [V] \mapsto \dim_\R V^0 - \dim_\R V^1\in\Z$. 
A generator of the cyclic group $A_{0,0}$ is given by the class of the
graded vector space $V= V^0\oplus V^1=\R \oplus 0$.

\subsubsection{ $r=s\ge 1$ }\label{sss.rs}
Using $r=s\geq 1$ allows us to deal with the case $s-r=0$
in the ungraded setting, where the definition would have us deal with representations of 
`$\Cl_{0,-1}$'.  We have
\[\begin{split}
  \Cl_{r,r-1} &\simeq \Mat(2^{r-1}, \Cl_{1,0} =\R\oplus \R ), \\ 
      \Cl_{r,r}   & \simeq \Mat( 2^{r-1}, \Cl_{1,1}=\Mat(2,\R))
          \simeq \Mat(2^{r-1}, \Cl_{2,0}) \simeq \Cl_{r+1,r-1}.
  \end{split}
\]
We recall that $A_{r,r}=\sM_{r,r-1}/\sM_{r,r}\simeq B_{r,r}=\sM_{r,r-1}/\sM_{r+1,r-1}$ and 
since $\Cl_{r,r}=\Cl_{r+1,r-1}=\Mat(2^r,\R)$ the two quotients are just equal. 
Let $V$ be a finite dimensional $\Cl_{r,r-1}$--module. 

Note that (cf. \cite[Theorem 5.7]{LM89}, Sec.~\ref{s.NHCT}) $\Cl_{r,r-1}$
has two irreducible representations.
The volume element $\go_{r,r-1}=e_1\cldots e_r\cdot f_1\cldots f_{r-1}$
is central and satisfies $\go_{r,r-1}^*=\go_{r,r-1}$, $\go_{r,r-1}^2=I$.  The two irreducible
representations $\varrho_{\pm}$ of $\Cl_{r,r-1}$ are
$2^{r-1}$--dimensional and they are characterized by
$\varrho_{\pm}( \go_{r,r-1} ) = \pm I$. Hence their multiplicities in $V$ are
given by $2^{1-r}\cdot\dim_\R\ker(\go_{r,r-1} \mp I)$.  
If $V$ carries a $\Cl_{r+1,r-1}\simeq\Cl_{r,r}$ structure, then
since $\go_{r,r-1}$ must anti-commute with the additional Clifford
generator, the two multiplicities must coincide. Conversely, if these
two multiplicities coincide, $E_{r+1}$ resp. $F_r$ may be constructed.

Thus the natural identification of $A_{r,r}=B_{r,r}$ with $\Z$ is
given by
\begin{equation}
  \tau_{r,r}:
  A_{r,r}\ni [V]\mapsto \frac 1{2^{r-1}}\cdot\Bl \dim_\R \ker( \go_{r,r-1}-I)
  -\dim_\R\ker(\go_{r,r-1} + I) \Br \in \Z.
\end{equation}
E.g., a positive generator of the cyclic group $A_{1,1}$
is given by the $\Cl_{1,0}$ module $(\R, E_1=I)$.

In general, a positive generator of the cyclic group $A_{r,r}$
is given by $\R^{2^{r-1}}$ together with a choice of Clifford
matrices $E_1,\ldots,E_r$, $F_1,\ldots,F_{r-1}$ such that
$E_1\cldots E_r\cdot F_1\cldots F_{r-1} = I$.

Since $\Cl_{r,r+7}\simeq \Mat(16,\Cl_{r,r-1})$, up to some tweaking
with matrix algebras, the previous analysis carries over to 
all $s\ge r\ge 1$ with $s-r\equiv 0\mlmod 8$. Details are left
to the reader.

\subsubsection{ $s=r+1$ }\label{sss.rsone}
A finite dimensional ungraded real
$\Cl_{r,r}=\Mat(2^r,\R)$--module $V$ carries a 
$\Cl_{r,r+1}\simeq \Mat(2^r,\Cl_{0,1}\simeq \C)$--module structure if and only if 
$V$ decomposes into an \emph{even} number of irreducible
$\Cl_{r,r}$--modules. Thus
\begin{equation}
  \tau_{r,r+1}: A_{r,r+1} \ni [V]\mapsto \frac 1{2^r} \dim_\R V \mlmod 2\in \Z/2\Z.
\end{equation}
A generator of the cyclic group $A_{r,r+1}$ is given by
the (up to equivalence unique) $2^r$--dimensional $\Cl_{r,r}$--module,
i.e. $\R^{2^r}$ with a choice of Clifford matrices $\Edots$,
$F_1,\ldots,F_{r}$. Given such a module,
there is a unique choice of another Clifford generator $E_{r+1}$ such
that $E_1\cldots E_{r+1}\cdot F_1\cldots F_r=I$. Thus, the module
may at the same time be viewed as a generator of $A_{r+1,r+1}$.

\subsubsection{ $s=r+2$ }\label{sss.rstwo}
Let $(V,F_1)$ be a finite dimensional
ungraded $\Cl_{0,1}$--module. Then $F_1$ is a complex structure and 
there exists $F_2\in\sB(V)$ with $F_2^*=-F_2$, $F_1 F_2=-F_2F_1$ 
if and only if $\dim_\C V = \frac 12 \dim_\R V\equiv 0\mlmod 2$. Thus
the map of the table is given by $A_{0,2}\ni [V]
\mapsto \dim_\C (V,i=F_1)\mlmod 2 = \frac 12 \dim_\R V\mlmod 2$.
A generator of the cyclic group $A_{0,2}$ is given by
$\bl\R^2, L_1 =\twomat 0 {-1}1 0\br$.

For arbitrary $r$, cf. the previous discussion, the map is given
by
\begin{equation}
  \tau_{r,r+2}:A_{r,r+2} \ni [V]\mapsto \frac 1{2^{r+1}} \dim_\R V \mlmod 2\in \Z/2\Z.
\end{equation}

\subsubsection{ $s=r+4$, cf. the case $r=s$ above}
\label{sss.rsfour} This case is mostly parallel to Sec.~\ref{sss.rs}
above as $\Cl_{r,r+3}\simeq \Mat(2^r,\HH\oplus\HH)$ has two
$2^{r+2}$--dimensional irreducible representations $\varrho_\pm$
which are distinguished by $\varrho_{\pm}( \go_{r,r+3} ) = \pm I$.
Hence their multiplicities in a
$\Cl_{r,r+3}$--module $V$ are given by
$2^{-r-2}\cdot\dim_\R\ker(\go_{r,r+3} \mp I)$.  
Again, $V$ carries a $\Cl_{r,r+4}$ structure if and only if these multiplicities
coincide. In sum, the natural identification of $A_{r,r+4}$ with $\Z$ is given by
\begin{equation}
  \tau_{r,r+4}:A_{r,r+4}\ni [V]\mapsto 
  \frac 1{2^{r+2}}\Bl \dim_\R \ker(\go_{r,r+3}-I)
  -\dim_\R\ker(\go_{r,r+3}+I)\Br \in\Z.
\end{equation}
In applications to topological phases, where spaces are typically taken to be complex, 
the above quaternionic index is often computed with complex dimension and 
takes range in $2\Z$, see~\cite{GSB, KatsuraKoma} for example.

We note an important consequence of the previous discussion.

\begin{theorem}\label{t.forgetful}
There are natural forgetful maps 
$\fg_{r+1,s}:A_{r+1,s}\to  A_{r,s}$ sending the
  $\Cl_{r+1,s-1}$--module $V$ to the underlying $\Cl_{r,s-1}$--module.

\subsubsection*{1.} For $\fg_{2,2}$ we have 
$\tau_{1,2}\bl \fg_{2,2}([V]) \br = \tau_{2,2}([V]) \mlmod 2$, hence
$\fg_{2,2}$ is surjective.

\subsubsection*{2.} Furthermore, $\fg_{1,2}$ is an isomorphism. 
More concretely,
  \[
    \tau_{0,2}\bl \fg_{1,2}([V]) \br = \tau_{1,2}([V]).
   \] 

Explicitly, the module $(\R^2, K_1,K_2,L_1)$ generates $\KO_j$,
$j=0,1,2$ in the following incarnations:
  \begin{enumerate}\renewcommand{\labelenumi}{\textup{(\roman{enumi})}}
\item  $\tau_{2,2}$ of the $\Cl_{2,1}$--module $(\R^2, K_1,K_2,L_1)$
equals $1\in\Z$. 
      
\item $\fg_{2,2}( \R^2, K_1,K_2,L_1) = (\R^2, K_1, L_1)$
  which generates $A_{1,2}$ and hence 
      \[\tau_{1,2}(\R^2, K_1, L_1) = 1\in\Z/2\Z.\]
      
\item Finally, $\fg_{1,2}( \R^2, K_1, L_1) = (\R^2, L_1)$
  which in turn generates 
      $A_{0,2}$ as \[\tau_{0,2}( \R^2, L_1 ) = 1\in \Z/2\Z.\]
    \end{enumerate}      
\end{theorem}
\begin{proof} The proof follows by walking through the cases~\ref{sss.rs}--\ref{sss.rstwo} above.
\end{proof}

%*********************************************************************************
\subsection{The $\Cl_{r,s}$ Fredholm index}
\label{ss.ClrsFred}
%*********************************************************************************

\begin{dfn}\label{p.ClrsFred.1} Let $H$ be a $\Clrs$--Hilbert space. 
Following Atiyah and Singer\footnote{      % FOOTNOTE
Cf. also \cite{Kar70}.  The $\sF^k$ in \cite{AS69} is our $\sF^{0,k}$.
The adaption of \cite{AS69} to the two parameter case is
straightforward as, up to a canonical homeomorphism,
$\sB^{r,s+1}, \sF^{r,s+1}$ depend only on the
difference $s-r\mlmod 8$.}  %% END OF FOOTNOTE
and recalling the notation $\sB^{r,s+1}$ from \Eqref{eq.NHCT.1}, let 
\begin{equation}\label{eq.NHCT.2}
  \sF^{r,s+1}(H):=\bigsetdef{ T\in\sB^{r,s+1}(H)}{ T\text{ Fredholm} }.
\end{equation}  
Furthermore, if $H$ is standard then let $\sF^{r,s+1}_*(H)=\sF^{r,s+1}$ if
$r-s\not\equiv 2\mlmod 4$ and let $\sF^{r,s+1}_*(H)$ be the interesting path component
of $\sF^{r,s+1}$ if $r-s\equiv 2\mlmod 4$, see
\cite[p.~7]{AS69}.\footnote{ % FOOTNOTE 
When comparing, note that in \cite{AS69}
$r=0, s+1=k$, hence $r-s\equiv 2\mlmod 4$ if and only if $k\equiv 3\mlmod 4$.} % END FOOTNOTE
  Moreover, denote by $\FO_{*}^{r,s+1}\subset\sF_{*}^{r,s+1}$ (resp.
  $\FO^{r,s+1}\subset\sF^{r,s+1}$)  the space of those
$T$ with $\|T\|=1$ and whose image $\pi(T)\in \sQ(H):=\sB(H)/\sK(H)$
in the Calkin algebra is unitary. 
\end{dfn}

\begin{remark}
  The subspace $\FO^{r,s+1}$ ($\FO_*^{r,s+1}$)
  is a deformation retract of $\sF^{r,s+1}$ ($\sF_{*}^{r,s+1}$), in
particular the inclusion $\FO^{r,s+1}\subset \sF^{r,s+1}$ 
  ($\FO_{*}^{r,s+1}\subset \sF_{*}^{r,s+1}$) is
a homotopy equivalence.
\end{remark}

For $T\in\sF^{r,s+1}$ the kernel
$\ker T$ is naturally an ungraded $\Cl_{r,s}$--module and one therefore
puts
\begin{equation}
  \ind_{r,s+1}(T) = [\ker T]\in A_{r,s+1}, \quad \text{ and } 
  \Ind_{r,s+1}(T) := \tau_{r,s+1}\bl \ind_{r,s+1} (T) \br,
\end{equation}     
cf. \Eqref{eq.ABSC.16}. The index $\ind_{r,s+1}$ is locally constant on $\sF^{r,s+1}$, cf.
\cite[Sec.~III.10, Prop.~10.6]{LM89}. See also~\cite{GSB, KatsuraKoma} for a similar construction.

We note that $\ind_{r,s+1}$ is surjective. The following construction
is instructive and very explicit. Fix a standard $\Cl_{r,s+1}$--Hilbert
space $H_0$. Then, given a $\Cl_{r,s}$--module $V$, put $H=H_0\oplus V$
and $T_V := F_{s+1}\restr V\oplus 0\restr V$. Then
\begin{equation}
  \label{eq.ABSC.9}
  \ind_{r,s+1}(T_V) = [V] \in A_{r,s+1} \simeq \sM_{r,s}/\sM_{r,s+1}.
\end{equation}
Clearly, the only purpose of $H_0$ is to make $H$ itself standard.
Otherwise, one could just put $T_V=0$ on $H=V$ which has index $[V]$
as well.

Theorem~\ref{t.forgetful} has an immediate consequence for Fredholm
operators. 
\newcommand{\Tmatrix}{\begin{pmatrix} 
  0 & - T^* \\ T & 0\end{pmatrix} }
\newcommand{\tmatrix}{ \begin{matrix} 0 & - T^* \\ T & 0 \end{matrix} }
\newcommand{\TTmatrix}{\begin{pmatrix} 
  0 & T_{1,1} \\ T_{1,1} & 0
                       \end{pmatrix} }

\begin{prop}\label{p.ClrsFred.2}   
For a Fredholm operator $T\in \sF^{r,r}$ one has,
see \Eqref{eq.ABSC.13}, the index formula
\[
  \Ind_{r,r}(T) = \Ind_{r+1,r+1}\begin{pmatrix} 0 & - T^* \\ T &
  0\end{pmatrix} = \Ind_{r+1,r+1} \bl T_{1,1} \br ,
\]
where $T_{1,1}$, defined in \eqref{eq.ABSC.13}, 
is viewed as an element of $\sF^{r+1,r+1}$. 
This formula also holds true for $r=0$, where $T$ is just a Fredholm
  operator and $\Ind_{0,0}$ denotes the ordinary Fredholm index.

Next denote by $\fg_{r+1,s+1}: \sF^{r+1,s+1} \to \sF^{r,s+1}$
  the natural forgetful map. 
For $r\ge 1$ and $T\in\sF^{r+1,r+1}$ one has the index
formula
\[
  \Ind_{r,r+1}(\fg_{r+1,r+1}(T) ) = \Ind_{r+1,r+1}(T) \mlmod 2,
\]
and for $T\in\sF^{r,r+1}$
\[
  \Ind_{r-1,r+1}(\fg_{r,r+1}(T) ) = \Ind_{r,r+1}(T).
\] 
In particular, we have for a Fredholm operator $T$ with 
  $T_{1,1}=\Tmatrix$,
\begin{equation}
  \begin{split}
    \Ind_{0,2} \TTmatrix  
        &= \Ind_{1,2} \TTmatrix   \\
        &= \Ind_{2,2} \TTmatrix  \mlmod 2 \\
        &= \Ind_{1,1} \bl T_{1,1}  \br \mlmod 2
          = \Ind_{0,0}(T) \mlmod 2.
  \end{split}
\end{equation}   
\end{prop}  

\begin{remark}\label{p.ClrsFred.3} As explained in Remark~\ref{p.ICA.2}
there is a one to one correspondence $T\leftrightarrow T E_{r+1}$
between $\Cl_{r+1,s}$--anti-linear
skew-adjoint Fredholm operators and $\Cl_{s+1,r}$--anti-linear
self-adjoint Fredholm operators. The Fredholm index of $T$ lies
naturally in $A_{r+1,s+1}$ and is locally constant since in a
nontrivial eigenspace $\ker(T^2+\lambda^2)$ the phase of $T$ gives
an additional skew-adjoint
Clifford symmetry. Consequently, for a self-adjoint
$\Cl_{s+1,r}$--anti-linear operator $S(=T E_{r+1})$ the phase
of $S$ gives an additional self-adjoint Clifford symmetry.
Therefore, the index of $S$ lies naturally in $B_{s+1,r+1}$,
which by \Eqref{eq.ABSC.16} is isomorphic to $A_{r+1,s+1}$.
\end{remark}
  
%############################################
\section{Some useful homotopies and the Cayley transform}
\label{s.NHCT}
%############################################
%*********************************************************
%*********************************************************

%*********************************************************
\subsection{Spaces of complex structures}
\label{ss.SFOCS}
%*********************************************************

In this subsection we collect and expand on some of the homotopy theoretic
results from \cite{AS69}. 

\begin{dfn}\label{p.NHCT.1} Let $\sA$ be a real unital $C^*$--algebra
with Clifford elements $e_1,\ldots,e_r$, $f_1,\ldots, f_s\in\sA$.
Denote by $\sJ^{r,s+1}(\sA)$ the set of complex structures
\footnote{ A complex structure
$J$ is a unitary skew-adjoint element, in particular $J^2=-1_\sA$.}
%% END OF FOOTNOTE
$J\in\sA$ anti-commuting with the Clifford elements and by
$\sJ^{0,0}(\sA)$ the space of unitaries in $\sA$.

If $H$ is a $\Cl_{r,s}$--Hilbert space then 
we set $\sJ^{r,s+1}(H):= \sJ^{r,s+1}(\sB(H))$. 
\end{dfn}

If $H$ is a $\Clrs$--Hilbert space, $s>0$ or $r=s=0$, we put
\begin{align}
  \ovl{\Omega}_{r,s}(H) &:= \bigsetdef{T\in \sJ^{r,s}(H) }{ T-F_s \text{
     compact} },\label{eq.NHCT.3}\\
   \tilde\Omega_{r,s}(H) &:= \bigsetdef{T\in\sJ^{r,s}(H) }{ \|T-F_s\|_\sQ < 2},
      \label{eq.NHCT.4}
\end{align}  
cf. \cite[p.~19]{AS69}. In the case $r=s=0$ this must be read
with $F_0=I$.

Note that for $T\in \sJ^{r,s}(H)$ the element $\pi (F_s T)$ is unitary
and for a unitary $u$ one has the equivalence $\|u+I\|<2
\Leftrightarrow 1\not\in\spec u$ as $|z+1|\le 2$ for $z$ on the circle
with equality only for $z=1$.

For $T\in\sJ^{r,s}(H)$ this translates into the equivalence
\begin{equation}\label{eq.NHCT.5}
  T\in \tO_{r,s} \Leftrightarrow \| T-F_s\|_\sQ <2 
     \Leftrightarrow \| F_sT+I\|_\sQ <2 
     \Leftrightarrow 1 \not \in \specess\bl F_s T\br.
\end{equation}     

When it is necessary for disambiguation we will write 
\[
  \ovl{\Omega}_{r,s}(H; \Edots,\; \Fdots)
\]
to display explicitly the generators of the Clifford module structure.

%*********************************************************
\subsection{Some elementary homotopies}
\label{ss.EH}
%*********************************************************

\begin{prop}\label{p.NHCT.2} Let $\sA$ be a real unital $C^*$--algebra
with Clifford elements $e_1,\ldots,e_r$, $f_1,\ldots,f_s\in\sA$.
Then the diagonal in
\[
  \bigsetdef{ (J_0,J_1)\in \sJ^{r,s+1}(\sA)\times \sJ^{r,s+1}(\sA)}{ \| J_0-J_1 \| <2 }
\]
is a strong deformation retract. The result is also valid in the corner case $r=s+1=0$.
\end{prop}
\begin{proof} The result is proved for $r,s\ge 0$, the corner case
$r=s+1=0$ is left to the reader. The retraction will be given
explicitly. Given
$J_0,J_1\in\sJ^{r,s+1}(\sA)$ with $\| J_0-J_1 \| < 2$, put
 $   Z:= Z(J_0,J_1) = 1_\sA - J_0 J_1$.
Then $Z$ commutes with $e_1,\ldots, e_r, f_1,\ldots,f_s$ and 
\begin{equation}  \label{eq.NHCT.6}
    Z J_1 = J_0+J_1 = J_0 Z.
\end{equation}
Taking adjoints we have $J_1 Z^* = Z^* J_0$, hence
\begin{equation}  \label{eq.NHCT.7}
    ZZ^* J_1 = Z^* J_0 Z = J_1 Z^* Z.
\end{equation}
We claim that $Z_x:= 1_\sA - x J_0 J_1$ is invertible 
for $0\le x \le 1$. Indeed
\[
      Z_x = 1_\sA - x(J_0-J_1) J_1 - x J_1^2 =
      (1+x) \bl 1_\sA - \frac x{1+x} (J_0-J_1) J_1 \br,
\]
which is invertible since
\[
    \bigl \| \frac x{1+x} (J_0-J_1) J_1 \bigr \| < \frac{2x}{1+x} \le
    1.
\]
In view of this and \Eqref{eq.NHCT.6}, \eqref{eq.NHCT.7} the
retraction is now given by
\[
  (J_0, J_1) \mapsto \Bl J_0, Z_x |Z_x|\ii J_1 \bl Z_x |Z_x|\ii \br^*
  \Br.\qedhere
\]
\end{proof}  

\begin{lemma}\label{p.NHCT.3} Let $H$ be a $\Cl_{r,s}$--Hilbert space.
The neighbourhood of $F_{s}\in\tO_{r,s}(H)$ given by
\begin{equation}
  \begin{split}\label{eq.NHCT.8}
  \tO_{r,s,*} :&=  \bigsetdef{\go\in\tO_{r,s}}{ \|\go - F_{s}\|<2 }\\
      &= \bigsetdef{\go\in\tO_{r,s}}{ 1\not\in\spec\bl F_{s}\go\br}
  \end{split}
\end{equation}
is contractible.

Furthermore, with the map
$\Phi_{r,s+1}: T\mapsto -F_{s} e^{\pi T F_s}$ of Theorem
\ref{p.SCHE.4} we have for invertible $T\in \FO^{r,s+1}$ the
inequality $\|\Phi_{r,s+1}(T) - F_{s} \| <2 $; that is, 
$\Phi_{r,s+1}(T)\in\bO_{r,s,*}$.  
Here, analogously to \Eqref{eq.NHCT.8}, 
$\bO_{r,s,*} := \bigsetdef{\go\in\bO_{r,s}}{ \|\go - F_{s}\|<2 }$,
cf. \Eqref{eq.NHCT.3}.
\end{lemma}
\begin{proof} The second equality in \Eqref{eq.NHCT.8} follows analogously
to the equivalence \Eqref{eq.NHCT.5}.
That $\tO_{r,s,*}$ is contractible follows as in the proof of 
 Prop.~\ref{p.NHCT.2}.

For invertible $T\in \FO^{r,s+1}$ the operator $TF_{s}$ is skew,
invertible and of norm $1$, hence
\[
   \spec(TF_{s}) \subset [-i,0)\cup (0,i]
\] 
and thus by the spectral mapping theorem
\[
    \spec( e^{\pi T F_{s}} ) \subset S^1\setminus \{1\} = 
      \bigsetdef{z\in S^1}{ |z+1|<2 },
\]
so consequently
\[
  \bigl\|\underbrace{ - F_{s} e^{\pi T F_{s} }}_{\Phi_{r,s+1}(T)} - F_{s} \bigr\| < 2.
    \qedhere
\]
\end{proof}

The following generalization of Theorem~\ref{p.SCHE.4} is crucial
for the theory of Fredholm pairs and the spectral flow.

\begin{theorem} \label{p.NHCT.4}
Let $f:i\R\to i\R$ be an increasing continuous function
such that
\begin{thmenum}
  \item $f$ is strictly increasing in a neighbourhood of $0$,
  \item $|f(x)|\le 1$ for all $x\in i\R$,  
  \item $f(-x) = -f(x) = \ovl{f(x)}$ for all $x\in i\R$.
\end{thmenum}
Furthermore, let $H$ be a standard $\Cl_{r,s+1}$--Hilbert space.
Then the map
\[
  \Phi_f: \sFs^{r,s+1} \to \widetilde{\Omega}_{r,s},
    \quad T\mapsto -F_s e^{\pi f(T F_s) }
\]
is a homotopy equivalence. 
  Furthermore, if $T\in\sFs^{r,s+1}$ is  invertible then 
$\Phi_f(T)\in \tO_{r,s,*}$.

A particularly nice homotopy equivalence is given by the Cayley
transform
\begin{equation}\label{eq.NHCT.9}
  \Phi: \sFs^{r,s+1} \to \widetilde\Omega_{r,s}, 
       \quad T\mapsto -F_s (I+TF_s)(I-TF_s)\ii.
\end{equation}
\end{theorem}
\begin{remark}\label{p.NHCT.5}\indent\par
\paragraph*{1} We first clarify the last remark about the Cayley transform, i.e.
  \begin{equation}\label{eq.NHCT.10}
  e^{\pi f(x)} =: \frac {1+x}{1-x},
\end{equation}
or\footnote{%                FOOTNOTE
The formula in parenthesis on the right of
  \Eqref{eq.NHCT.11} holds only for $|x|<1$.}  % END FOOTNOTE
  \begin{equation}\label{eq.NHCT.11}
  f(x) = \frac 1\pi \log\frac{1+x}{1-x} 
       \Bigl( = \frac i\pi \arctan \frac{2 x/i}{1+x^2} \Bigr), \quad x\in i\R.
\end{equation}
Thus indeed
\[
    -F_{s} e^{\pi f( T F_{s})} = -F_{s} (I+ T F_{s})(I- T F_{s})\ii.
\]

\paragraph*{2. We comment on the case \emph{$r=0$, $s=0$}.}
Then $\tO_{0,0}$ is the space of unitaries $U$ with 
$\|U-I\|_{\sQ}<2$ resp. $-1\not\in\specess(U)$
and the Cayley transform is $A\mapsto (I+A)(A-I)\ii$. In the complex
case, where one considers self-adjoint instead of skew-adjoint
operators, $A$ should be replaced by $iA$. This should be compared to
\cite{BLP05} and \cite[Sec.~6.1]{KirLes:EIM}.
\end{remark}

\begin{proof}[Proof of Theorem~\ref{p.NHCT.4}]
This is essentially a consequence of Theorem~\ref{p.SCHE.4}.
    
  We first check that $\Phi_f$ maps $\sFs^{r,s+1}$ into $\widetilde\Omega_{r,s}$.
Abbreviate $g(x):= e^{\pi f(x)}$. Then $g$ maps $i\R$ into $S^1$ and
  $g(-x) = g(x)\ii = \ovl{g(x)}$.  Thus for $T\in\sFs^{r,s+1}$ we have
\begin{align*}
  \bl F_{s} g(T F_{s}) \br^* & = - \ovl g(T F_{s} ) F_{s} \\
       & = - g( - T F_{s} ) F_{s} = - F_{s} g( TF_{s} ),
\end{align*}
hence $\Phi_f(T F_{s})$ is skew-adjoint and, since $g$ takes
values in $S^1$, unitary. Consequently, $\Phi_f(T F_{s})^2=-I$.
Clearly, since both $T$ and $F_{s}$ anti-commute with the Clifford
generators except $F_{s}$ it follows that  
$\Phi_f(T F_{s}) \in \sJ^{r,s}(H)$.

Finally, since $T F_{s}$ is a Fredholm operator, there is a gap 
$[-i\eps, i\eps]$ in the essential spectrum. The properties (1)-(3) of
$f$ imply that there is a small arc $\{ |\operatorname{arg} z| <\delta \}$
around $1\in S^1$ which does not meet the essential
spectrum of $e^{\pi f(T F_{s})}$.  Therefore,
\[
  \| - F_{s} g(T F_{s}) - F_{s}\|_\sQ = \| g(T F_{s} ) + I\|_{\sQ} <2,
\]
proving that $\Phi_f$ maps into $\widetilde\Omega_{r,s}$ as claimed.

Now define the function
\begin{equation}\label{eq.NHCT.12}
     \Psi: i\R \to i\R,\quad x\mapsto \begin{cases} x,& |x|\le 1,\\
       x/|x|,& |x|\ge 1.
     \end{cases}
\end{equation}
Then $f_u (x) := (1-u) f(x) + u \Psi(x), 0\le u\le 1$ is a homotopy of
functions satisfying (1)-(3) and consequently $\Phi_{f_u}$ is a
homotopy between $\Phi_f$ and $\Phi_{\Psi}$.

On the deformation retract $\FO_*^{r,s+1}\subset \sFs^{r,s+1}$ the map $\Phi_\Psi$
coincides with the map $\Phi_{r,s+1}$ of Theorem~\ref{p.SCHE.4}. Thus
$\Phi_\Psi$ and hence $\Phi_f$ are homotopy equivalences
as claimed.

Finally, if $T$ is invertible then one argues as in the
proof of Lemma~\ref{p.NHCT.3}. Namely, $f(TF_{s})$ is invertible
as well and hence $1\not\in\spec\bl e^{\pi f( T F_{s})}\br$,
thus
  \[
    \|-F_{s} e^{\pi f( T F_{s})}- F_{s}\|<2.
    \qedhere
  \]
\end{proof}

%#######################################################
\section{Fredholm pairs and the Clifford index}
\label{s.FPCI}
%#######################################################

Here we introduce a Fredholm theory of pairs of complex structures
where the index naturally takes values in
$A_{r,s}=\sM_{r,s-1}/\sM_{r,s}\simeq \KO_{s-r}(\R)$, cf.  Sec.
\ref{s.ABSC}. The complex analogue has a long history. Kato
systematically studied Fredholm pairs of subspaces in a Banach space
\cite[IV.4.1]{Kat:PTL}. The first hint that in the Hilbert space
setting one should instead look at the corresponding orthogonal
projections as the primary objects can be found in \cite[Rem.
4.9]{BDF73}. These projections appear prominently in the theory of
boundary value problems for Dirac type operators \cite[Sec.
24]{BooWoj:EBP}. The theory was rediscovered  and further developed
 in the influential 
paper \cite{ASS94} by Avron, Seiler, and Simon;
see also \cite{BL01}. For further details see Sec.~\ref{sss.IPP}
and Sec.~\ref{sss.CC}.

%*********************************************************
\subsection{Fredholm pairs of complex structures}
\label{ss.FPCS}
%*********************************************************

\begin{dfn}\label{p.FPCI.3} Let $H$ be a $\Cl_{r,s}$--Hilbert space.
  A pair $(J_0, J_1)$ of elements of $\sJ^{r,s+1}(H)$ is
called a \emph{Fredholm pair} if $\|J_0-J_1\|_\sQ <2$.

Recall from Sec.~\ref{ss.PN} that $\|\ldots \|_\sQ$ is an abbreviation
for $\|\pi(\ldots)\|_\sQ$.
\end{dfn}

Given a Fredholm pair $(J_0, J_1) \in \sJ^{r,s+1}(H)\times \sJ^{r,s+1}(H)$
we note that $H$ becomes a $\Cl_{r,s+1}$--Hilbert space by setting $F_{s+1}=J_0$.
Moreover, the space of those $J_1$ with $(J_0, J_1)$ being a Fredholm
pair is nothing but the space 
\[
  \tilde\Omega_{r,s+1}(H;E_1,\ldots,E_r,
   F_1,\ldots,F_{s},F_{s+1}=J_0),
\]   
see \Eqref{eq.NHCT.3}.

From now on we fix a $\Cl_{r,s}$--Hilbert space $H$. The useful
identities of the next Lemma are straightforward to check.

\begin{lemma}[{\cite[Lemma 5.3]{CPSchuba19}, \cite[Theorem 2.1]{ASS94}}]
\label{lem:T0_T1_identities}
  Let $J_0,J_1 \in\sJ^{r,s+1}(H)$. Define the skew-adjoint operators 
\[
   T_0 = \frac{1}{2}(J_0 + J_1), \qquad T_1 = \frac{1}{2}(J_0 - J_1).
\]
Then $T_0^2 + T_1^2 = -1$, $T_0T_1 = -T_1 T_0$, $T_0J_0 = J_1 T_0$,
$T_0J_1 = J_0 T_0$, $T_1J_0 = -J_1 T_1$, and $T_1J_1 = -J_0 T_1$.
Furthermore, both $T_0$ and $T_1$ anti-commute with the 
Clifford generators $\Edots$ and $\Fdots$.
\end{lemma}

\begin{defprop} [Definition of $\ind_{r,s+2}^1(J_0, J_1)$] \label{p.FPCI.5}
Let $(J_0, J_1)$ be a Fredholm pair in $\sJ^{r,s+1}(H)$. Then 
for all $t\in [0,1]$, the operator $\gamma(t)= (1-t) J_0 + t J_1 \in\sF^{r,s+1}$
is a skew-adjoint Fredholm operator anti-commuting with 
$\Edots$, $\Fdots$, and 
\[
  \Phi_{r,s+1}\bl \gamma(t) \br := - F_{s} e^{\pi \gamma(t) F_{s} },
   \quad 0\le t \le 1
\]
is a loop in $\tilde\Omega_{r,s}(H)$.

Now embed the $\Cl_{r,s+1}$--Hilbert space 
$(H; \Edots,\; \Fdots,F_{s+1}:=J_0)\hookrightarrow H'$ into a standard
$\Cl_{r,s+1}$--Hilbert space and put $J_j':= J_j\oplus F_{s+1}\restr{ H^\perp }$.
Then we define $\ind_{r,s+2}^1(J_0, J_1)$ to be the class of the loop 
$\Phi_{r,s+1}\circ \gamma\oplus {F_{s}}\restr{H^\perp}
 =\Phi_{r,s+1}\circ \bl \gamma\oplus {F_{s+1}}\restr{H^\perp} \br
$
in $\pi_1\bl \tilde\Omega_{r,s}(H'),F_s \br$,
which by Theorem~\ref{p.SCHE.6} is canonically isomorphic to 
the abelian group 
$A_{r,s+2}=\sM_{r,s+1}/\sM_{r,s+2}\simeq \KO_{s+2-r}(\R)$.

Alternatively and equivalently, $\ind_{r,s+2}^1(J_0, J_1)$ denotes the
stable homotopy class of the loop $\Phi_{r,s+1}\circ \gamma$
in the stable fundamental group 
$\pi_1^{\textup S}(\tilde\Omega_{r,s}(H), F_{s})$ 
which again is canonically isomorphic to $A_{r,s+2}$,
see Sec.~\ref{ss.SHFC}.
\end{defprop}
\begin{proof} By construction $\gamma(t)$ is skew-adjoint and
anti-commutes with the Clifford generators $\Edots$, $\Fdots$.  To see
that it is a Fredholm operator we note that
\begin{equation}  \label{eq.FPCI.4}
\bigl \| -J_1 \gamma(t) - I\bigr\|_\sQ =
     \bigl \| \gamma(t) - J_1\bigr\|_\sQ <1
     \text{ for } \frac 12\le t\le 1,
\end{equation}
and similarly $\bigl \| J_0 \gamma(t) - I\bigr\|_\sQ <1$
for $0\le t \le \frac 12$. 

  To prove the claim about $\Phi_{r,s+1}\circ\gamma$ we note that
  $\Phi_{r,s+1}(\gamma(0))=\Phi_{r,s+1}(\gamma(1)) = F_{s}$. From
\Eqref{eq.FPCI.4} it follows that for all $0\le t\le 1$, 
$\gamma(t) F_{s}$ lies in the open $1$ ball around a unitary of
square $-1$ ($J_0 F_{s}$ resp.  $J_1 F_{s}$). Thus 
$\specess\bl \gamma(t) F_{s}\br\subset [-i,0)\cup(0,i]$
and hence $1\not \in \specess\bl e^{\pi\gamma(t) F_{s}}\br$.
  Consequently, cf. \Eqref{eq.NHCT.5}, 
  $\| \Phi_{r,s+1}(\gamma(t)) - F_{s}\|_\sQ <2$ and $\gamma$
is a loop in $\tilde\Omega_{r,s}$ based at $F_{s}$ as claimed.
\end{proof}

It is an immediate consequence of the construction that
$\ind_{r,s+2}^1$ is locally constant on the space of Fredholm pairs
$(J_0, J_1)\in\sJ^{r,s+1}(H) \times \sJ^{r,s+1}(H)$.

We provide an alternative approach to the index of a Fredholm pair,
whose consequence will be the important additivity formula Prop.~\ref{p.FPCI.8} below.

\begin{prop}\label{p.FPCI.6} Let
$(J_0, J_1)\in \sJ^{r,s+1}(H)\times \sJ^{r,s+1}(H)$ be a Fredholm
pair and let $T_0, T_1$ as in Lemma~\ref{lem:T0_T1_identities}.
Then  
\begin{enumerate}
  \item The operator $\frac{1}{2}(J_0 + J_1)$ is Fredholm, 
  \item The space $\ker(J_0+J_1)$ is an ungraded $\Cl_{r,s+1}$--module
    w.r.t. the Clifford algebra generated by $\Edots$, $\Fdots$,
    $F_{s+1}:=J_0$.
  \item For any fixed $1>\lambda>0$, the spectral subspace 
  $X_\lambda = \chi_{(0, \lambda^2)}( -T_0^2 )$
carries naturally the structure of an ungraded $\Cl_{r,s+2}$--module with
    respect to the generators $\Edots$, $\Fdots$, $F_{r+1}=J_0$ and
    $F_{s+2}$ being the phase of
$J_0 T_1 T_0 = -\frac 14 (J_1+J_0J_1J_0)$, cf. Lemma~\ref{lem:T0_T1_identities}.
\end{enumerate}
\end{prop}
\begin{proof}
We note that this is essentially \cite[Lemma~5.3 and
Prop.~5.2]{CPSchuba19}. We will make repeated use of the identities
given in Lemma~\ref{lem:T0_T1_identities}.

\subsubsection*{(1)} was proved in the previous proof.

\subsubsection*{(2)} $\ker(J_0+J_1)$ is certainly invariant under
$\Edots$, $\Fdots$, and from $T_0 J_0 = J_1 T_0$ we infer that
it is invariant under $J_0$ as well.

\subsubsection*{(3)} For part (3), we first note that $X_\lambda$ is 
$T_1T_0$-invariant, $E_j T_1T_0 = T_1T_0 E_j$ and $J_0T_1T_0 = -T_1T_0 J_0$. 

Also observe that for $|\lambda|<1$ we have $-1<T_0^2|_{X_\lambda}<0$ (the
upper inequality is the Fredholm property), and so $T_0^2+T_1^2=-1$
tells us that $T_1$ is invertible on $X_\lambda$. Furthermore,
$J_0T_1T_0$ is skew-adjoint and anti-commutes with $\Edots$, $\Fdots$,
and with $J_0$. Finally, its phase $S$ is another complex structure
anti-commuting with $\Edots$, $\Fdots$, $J_0$, and hence $X_\lambda$
carries the structure of a $\Cl_{r,s+2}$--module with generators
$\Edots$, $\Fdots$, $J_0, S$.
\end{proof}

In light of this proposition we put for a Fredholm pair $(J_0, J_1)$
\[
   \ind_{r,s+2}^2(J_0, J_1):= \bigl[ \ker(J_0+J_1) \bigr]
      \in\sM_{r,s+1}/\sM_{r,s+2},
\]
i.e. the class of the $\Cl_{r,s+1}$--module $\ker(J_0+J_1)$ in
$A_{r,s+2}\simeq\KO_{s+2-r}(\R)$. 

It follows from (3) (cf. \cite[Prop.~5.2]{CPSchuba19}) that
$\ind_{r,s+2}^2$ is locally constant on the space of Fredholm pairs as
well.

\begin{theorem}\label{p.FPCI.7} The two definitions of a $\Cl_{r,s+2}$
index of a Fredholm pair $(J_0, J_1)$ coincide,
$\ind_\bullet^1=\ind_\bullet^2$. Hence from now on we may just write
$\ind_{r,s+2}(J_0, J_1)$.
\end{theorem}
As for the Fredholm index we write $\Ind_{r,s+2}(J_0,J_1):=
\tau_{r,s+2}\bl \ind_{r,s+2}(J_0,J_1) \br$.
\begin{proof} By the local constancy of both indices it suffices to check
the claim for fixed $J_0\in\sJ^{r,s+1}(H)$ on representatives of the
path components of the space
\[
  \bigsetdef{ J\in\sJ^{r,s+1}(H)}{ \| J-J_0\|_\sQ < 2 }\simeq 
    \tilde\Omega_{r,s+1}(H;\Edots,\; \Fdots, F_{s+1}=J_0)
\]
in a standard $\Cl_{r,s+1}$--Hilbert space. By doubling the space and
replacing $J_0$ by $J_0\oplus -J_0$ if necessary we may assume that
the Hilbert space is standard  with $J_0$ being one of the
Clifford generators.
Note that both indices are not affected by stabilizing:
  if we add another Hilbert space $H'$ and fix $J_0'\in \sJ^{r,s+1}(H')$
then $\ind_{r,s+2}^\ga(J_0, J_1) = \ind_{r,s+2}^\ga(J_0\oplus J_0',
  J_1\oplus J_0')$,  $\ga=1,2$.

Then by Cor.~\ref{p.SCHE.5} we know the path components of
\[ 
  \tilde\Omega_{r,s+1}(H; \Edots,\; \Fdots,J_0)
\]
and 
hence we may consider 
$J_1 = {J_0}\restr{H_0} \oplus -{J_0}\restr{V}$ where 
$H=H_0\oplus V$ is a $\Cl_{r,s+1}$--linear decomposition (generators $\Edots$, $\Fdots$,
  $F_{s+1}=J_0$)
 and $\dim V<\infty$.
Then $\frac 12 (J_0+J_1) = {J_0}\restr{H_0} \oplus 0\restr{V}$, whence
  $\ind_{r,s+2}^2(J_0, J_1) = [V]\in\sM_{r,s+1}/\sM_{r,s+2}$.  Further, 
\begin{equation}
\gamma(t) = (1-t) J_0 + t J_1 = {J_0}\restr{H_0} \oplus (1-2t)
{J_0}\restr{V}.
\end{equation}
On $H_0$ we have 
  $\Phi_{r,s+1}\bl \gamma(t) \br = - F_{s} e^{\pi J_0 F_{s} } \equiv F_{s}$,
and on $V$ we have
\[
  \Phi_{r,s+1}(\gamma(t) ) = - F_{s} e^{\pi (1-2t) J_0 F_{s}} =
   F_{s} e^{-2\pi t J_0 F_{s}}.
\]
By Theorem~\ref{p.SCHE.6} the homotopy class of this loop is again
$[V]$, thus $\ind_{r,s+2}^1(J_0,J_1) =\ind_{r,s+2}^2(J_0,J_1)$.
\end{proof}

The proof and the known standard generators of $\sM_{r,s+1}/\sM_{r,s+2}$ 
immediately show that (cf. Sec.~\ref{ss.ABS-examples})

\begin{cor}\label{p.FPCI.6A} For a standard $\Cl_{r,s+1}$--Hilbert space $H$ the
  map 
\[
  \ind_{r,s+2}(F_{s+1},\cdot): \tO_{r,s+1}(H)\to A_{r,s+2} \simeq \sM_{r,s+1}/\sM_{r,s+2}
  \]  
labels the connected components of $\tO_{r,s+1}(H)$.
For a non-standard $\Cl_{r,s+1}$--Hilbert space $\ind_{r,s+2}$ labels the stable
connected components.
\end{cor}

\newcommand{\indkp}{\ind_{r,s+2}}
\begin{prop}\label{p.FPCI.8} Suppose that $(J_0,J_1)$ and $(J_1,J_2)$ are
  Fredholm pairs of elements in $\sJ^{r,s+1}(H)$ with 
\begin{equation}\label{eq.FPCI.5}
\| J_0 - J_1 \|_\sQ<1 \text{ and } \| J_1 - J_2 \|_\sQ<1.
\end{equation}
Then
\[
\indkp(J_0,J_1) + \indkp(J_1,J_2)= \indkp(J_0,J_2).
\]
\end{prop}
\begin{proof} Denote by $\gamma_{01}, \gamma_{12},
\gamma_{02}$ the straight line paths from $J_0$ to $J_1$ etc.  The
assumption \Eqref{eq.FPCI.5} implies that the straight line homotopy
from $\gamma_{02}$ to the concatenation $\gamma_{12}*\gamma_{01}$ is a
  homotopy within paths in $\sF^{r,s+1}$ since any operator $T$ in the range
of the homotopy satisfies $\| T-J_0 \|_\sQ<2$. But then we have for
the homotopy classes
\[
  \bigl[ \Phi_{r,s+1}\circ\gamma_{02}\bigr]
  = \bigl[ \Phi_{r,s+1}\circ\gamma_{01}\bigr] +
    \bigl[ \Phi_{r,s+1}\circ\gamma_{12}\bigr]
\]
in $\pi_1(\tilde\Omega_{r,s}, F_{s} )$, thus
$\indkp(J_0,J_2)= \indkp(J_0,J_1)+ \indkp(J_1,J_2)$.
\end{proof}

%**********************************************************************
\subsection{Examples and comparison with the classical index of a pair
of projections}
\label{ss.CCIPP}
%**********************************************************************

Here we discuss a few illustrative cases for small $r,s$ and in
particular we connect our theory of Clifford covariant Fredholm pairs
to the classical theory of Fredholm pairs of orthogonal projections
in a real or complex Hilbert space.

\subsubsection{$r=0$, $s=0$} An element $J\in\sJ^{0,1}(H)$
is just a complex structure without any additional symmetries.
Thus, if $(J_0,J_1)$ is a Fredholm pair in $\sJ^{0,1}(H)$ then
by the discussion in Sec.~\ref{sss.rstwo} we infer
that $\Ind_{0,2}(J_0,J_1)=\frac 12 \dim\ker(J_0+J_1)\mlmod 2$.
Note that $\ker(J_0+J_1)$ is invariant under $J_0$ and 
is therefore automatically even, cf. \cite[Prop.~6.2]{CPSchuba19}.

We compare $\Ind_{0,2}$ to $\operatorname{Sf}_2$ of
\cite[Def.~2.1]{CPSchuba19}:

\begin{prop}\label{p.CPSpairindex}  Let $H$ be a finite dimensional Hilbert space,
$J\in\sJ^{0,1}(H)$, and $O\in \OMat(H)$ orthogonal.
Then
\[
  (-1)^{\Ind_{0,2}(J, O^*JO )} = \det(O),
\]   
thus $\Ind_{0,2}$ coincides with $\operatorname{Sf}_2$ of
\cite[Def.~2.1]{CPSchuba19}.
\end{prop}
\begin{proof} Both sides are locally constant in $\operatorname{O}$.
Furthermore, all $J\in\sJ^{0,1}(H)$ are unitarily equivalent. 
Therefore, it suffices to check this for $\OMat=I$ in which case
  both sides are $1$ and for 
\[
  \OMat=\twomat 0 1 1 0,\quad J=\twomat 0{-1}10,
\]
in which case both sides equal $-1$.
\end{proof}

\newcommand{\Ematrix}{\begin{pmatrix} 0&1\\1&0\end{pmatrix}}
\newcommand{\gradingmatrix}{\begin{pmatrix} 1&0\\0 &-1\end{pmatrix}}
\subsubsection{$r=1$, $s=0$} Let $H$ be a $\Clrs$--Hilbert space.
With regard to Prop.~\ref{p.ICA.1} and the second part of Remark~\ref{p.ICA.2}, 
we may write $H=H'\oplus H'$ such that $E_1=K_1$. If $J\in\sJ^{1,1}(H)$
then $J$ takes the form
\[
  J = \twomat 0 {-U^*}U 0 = \Re(U)\otimes L_1 + \Im(U)\otimes
  K_2.
\]
Thus $U=U(J)$ is a real unitary (aka orthogonal) operator.
If $(J_0,J_1)$ is a Fredholm pair in $\sJ^{1,1}(H)$ and if
$U_k:=U(J_k)$, $k=0,1$ denotes the associated unitary operators, then
comparison with Sec.~\ref{sss.rsone} shows that 
\begin{multline*}
  \Ind_{1,2}(J_0,J_1) =\frac 12 \dim\ker(J_0+J_1) \mlmod 2 \\
     = \dim\ker(U_0+U_1)\mlmod 2 
     = \dim\ker(I+U_0^*U_1)\mlmod 2.
\end{multline*}
Note that  $(J_0,J_1)$ being a Fredholm pair is equivalent
to $\|U_0-U_1\|_\sQ<2$ which in turn, cf. \Eqref{eq.NHCT.5},
is equivalent to $I+U_0^*U_1$ being Fredholm. Furthermore,
since $U_0, U_1$ are real, eigenvalues of $U_0^*U_1$ moving
through $-1$ come in pairs, hence $\dim\ker(I+U_0^*U_1)\mlmod 2$
has the stability properties of a Fredholm index.

We note that $\Ind_{1,2}(J_0,J_1)$ equals the \emph{parity
index} defined in \cite[Def. 3]{DSBW}. 
By
Theorem~\ref{t.forgetful} and Prop.~\ref{p.ClrsFred.2}
(or just directly) and suppressing forgetful maps from the notation, we have
\[
  \Ind_{1,2}(J_0,J_1) = \Ind_{0,2}(J_0,J_1).
\]

\begin{prop} Let $H$ be a finite dimensional Hilbert space,
$J_k=\twomat 0 {-U_k^*} {U_k} 0\in\sJ^{1,1}(H)$, $k=0,1$
with associated orthogonal matrices $U_0, U_1\in \OMat(H)$.
Then
\[
  (-1)^{\dim\ker(I+U_0^*U_1) }
    = \det(U_0^*U_1) = \bl \det U_0 \br
\cdot \bl \det U_1 \br.
\]
\end{prop}
This should be compared to \cite[Def.~1]{DSBW}.
\begin{proof} The proof is similar to the proof of Prop.~\ref{p.CPSpairindex}.  
\end{proof}  

\subsubsection{$r=2$, $s=0$: the index of a pair of projections}
\label{sss.IPP}
Let $H$ be a $\Clrs$--Hilbert space. Apply Prop.~\ref{p.ICA.1}
and write $H=H'\oplus H'$ such that 
\[
  E_1=K_1=\gradingmatrix,\quad E_2=K_2=\Ematrix.
\]
If $J\in\sJ^{2,1}(H)$ 
then $J$ takes the form
\[
  \begin{pmatrix} 0 & -(2P-I) \\ 2P-I & 0 \end{pmatrix},
\]    
with an orthogonal projection\footnote{A linear     % FOOTNOTE
operator $P\in\sB(H)$ is an orthogonal projection if $P^*=P=P^2$.}
                                               % END OF FOOTNOTE
$P\in\sB(H')$. Thus there is a one to one correspondence between
Fredholm pairs $(J_0,J_1)$ in $\sJ^{2,1}(H)$ and pairs of orthogonal
projections (without further symmetries) $P,Q\in\sB(H')$ with
$\|P-Q\|_{\sQ(H')}<1$.

Let us therefore recall that a pair of orthogonal projections
$P,Q$ acting on a separable real or complex Hilbert space $H$
is called a Fredholm pair if $\|P-Q\|_{\sQ}<1$.  This is equivalent to
the fact that $Q\restr{\ran P}: \ran P \to \ran Q$ is a Fredholm
operator. The index of this Fredholm operator is called the index
$\ind(P,Q)$ of the pair.  

\begin{lemma}[{cf. \cite[Prop.~3.1]{ASS94}} ]\label{l.CCIPP.1}
Let $P,Q$ be a Fredholm pair of orthogonal projections. 
Then $\ker (P+Q-I) = \ran P\cap \ker Q \oplus \ker P \cap \ran Q$
  and $\ind(P,Q)=\dim \bl \ran P\cap \ker Q \br
  - \dim \bl \ker P\cap \ran Q \br$.
\end{lemma}
\begin{proof} If $x\in\ker(P+Q-I)$ then $Px = (I-Q) x \in\ran P\cap \ker Q$
resp. $Qx = (I-P) x\in \ker P\cap\ran Q$, proving the inclusion $\subset$.
The inclusion $\supset$ is similarly easy. To see the index formula, 
one notes that the adjoint of the operator $Q\restr{\ran P}: \ran P
\to \ran Q$ is given by
$P\restr{\ran Q}: \ran Q\to \ran P$.\end{proof}

Let us return to  the $\Cl_{2,0}$--Hilbert space $H=H'\otimes \R^2$
as before and consider a Fredholm pair $(J_0,J_1)$ in $\sJ^{2,1}(H)$
with associated orthogonal projections $P_0,P_1\in \sB(H)$ such that
\[
  J_l =  \begin{pmatrix} 0 & -(2P_l-I) \\ 2P_l-I & 0 \end{pmatrix}, \;
    l=0,1.
\]
We have seen that $(J_0,J_1)$ being a Fredholm pair is equivalent
to $(P_0,P_1)$ being a Fredholm pair of orthogonal projections.
To compute the index $\ind_{2,2}(J_0,J_1)$ we apply
Sec.~\ref{sss.rs}, by first computing the volume element
of the $\Cl_{2,1}$ representation given by $E_1, E_2, J_0$,
\[
  \omega = E_1\cdot E_2 \cdot J_0 =
  \twomat 1 0 0 {-1} \cdot 
  \twomat 0 1 1  0   \cdot  
  \begin{pmatrix} 0 & -(2P-I) \\ 2P-I & 0 \end{pmatrix}
  =      
  \begin{pmatrix} 2P-I & 0 \\  0 & 2P-I \end{pmatrix}.
\]  
Furthermore in view of Lemma~\ref{l.CCIPP.1} we have

\[\begin{split}
  \ker (J_0+J_1) &=\begin{matrix} 
      \ran P\cap \ker Q \oplus \ker P\cap\ran Q \\
         \oplus \\
      \ran P\cap \ker Q \oplus \ker P\cap\ran Q \\
  \end{matrix}\\
     &= (\ran P\cap \ker Q)\otimes \R^2 
         \oplus (\ker P\cap\ran Q)\otimes \R^2.
\end{split}       
\]  
With respect to this decomposition, $\go$ acts 
on $\ker(J_0+J_1)$ as
\begin{equation}
  \omega \restr{\ker J_0+J_1} =  I\restr{(\ran P\cap \ker Q)\otimes \R^2}
        \oplus - I\restr{(\ker P \cap \ran Q)\otimes \R^2}.
\end{equation}
Now $\ind_{2,2}(J_0,J_1)$ is the class of the $\Cl_{2,1}$ module
$\ker (J_0,J_1)$ with Clifford generators $E_1, E_2, F_1=J_0$.
By Sec.~\ref{sss.rs} its class in $A_{2,2}$ 
is naturally identified with the number
\[ \frac 12 \bl \dim \ker (\omega\restr{\ker(J_0+J_1)} - I) 
        - \dim \ker (\omega\restr{\ker(J_0+J_1)}+I)  \br\in\Z.
\]
Thus
\[
\begin{split}
  \Ind_{2,2}(J_0, J_1) & = \frac 12 \dim\bl  \ker(J_0+J_1)\cap \ker
  (\omega - I)\br \\
       & \qquad     -\frac 12 \dim\bl \ker(J_0+J_1)\cap \ker (\omega  + I)\br \\
	    & = \dim \bl \ran P \cap \ker Q \br - \dim\bl \ker P\cap\ran Q\br \\
	    & =\ind (Q: \ran P\to \ran Q) =\ind (P,Q).
\end{split}
\]

The previous discussion also can be applied to 
the forgetful maps considered in 
 Theorem~\ref{t.forgetful} and Prop.~\ref{p.ClrsFred.2}. Namely, 
suppressing the forgetful maps from the notation, we have
\begin{equation}\label{eq.forgetfull-pairs}
    \Ind_{0,2}(J_0,J_1) = \Ind_{1,2}(J_0,J_1) =
    \Ind_{2,2}(J_0,J_1)\mlmod 2.
\end{equation}

\subsubsection{$r=0$, $s=0$ in the complex case}\label{sss.CC}
We briefly mention the case of a complex Hilbert space $H$ where
things are easier as there is a direct correspondence between
``complex structures'' and orthogonal projections since we may write
$J= i (2P-I)$. Therefore, there is no need to introduce additional
symmetries.

Given orthogonal projections $P$, $Q$ put $J_0:= i (2P-I)$, $J_1 = i (2Q-I)$. 
Then $J_0$, $J_1\in \sJ^{0,1}(H)$ and $\|J_0-J_1\|_{\sQ} = 2 \| P- Q\|_{\sQ}$. 
Hence $(J_0,J_1)$ is a Fredholm pair in the sense of Def.
\ref{p.FPCI.3} if and only if $(P,Q)$ is a Fredholm pair of
projections. Furthermore,
\begin{equation}
 \frac 12 ( J_0+J_1) = i (P+Q-I),
\end{equation}
thus by the Lemma~\ref{l.CCIPP.1} we have
\begin{equation}
    \ker J_0+J_1 = \ran P\cap \ker Q \oplus \ker P \cap \ran Q.
\end{equation}
With respect to this decomposition, $J_0$ acts as
\begin{equation}
    J_0\restr{\ker J_0+J_1} = i \cdot I\restr{\ran P\cap \ker Q}
        \oplus -i \cdot I\restr{\ker P \cap \ran Q}.
\end{equation}
The class $\ind_{0,2}(J_0,J_1)$ is the class of the $\CCl_1$--module
$\ker (J_0,J_1)$ with Clifford generator $E_1=J_0$. The class of
an ungraded $\CCl_1$--module $(V,E_1)$ in
$\sM_{0,1}/\sM_{0,2}\simeq \Z$ is naturally identified
(cf. Sec.~\ref{s.ABSC}) with
the number
\[ \dim \ker (E_1- i) - \dim \ker (E_1+i)\in\Z.\]
Thus
\[
\begin{split}
    \Ind_2(J_0, J_1) & = \dim\bl  \ker(J_0+J_1)\cap \ker (J_0 - i)\br
          - \dim\bl \ker(J_0+J_1)\cap \ker (J_0  + i)\br \\
	    & = \dim \bl \ran P \cap \ker Q \br - \dim\bl \ker P\cap\ran Q\br \\
	    & =\ind (Q: \ran P\to \ran Q) =\ind (P,Q).
\end{split}
\]

%***************************
\subsection{Standard forms of complex structures}
\label{ss.SP}
%***************************

Using the ABS construction of Sec.~\ref{s.ABSC}, we now give
standard generators of complex structures with a non-trivial index.
These pairs will be constructed to represent an arbitrary class in
$\KO_{s+2-r}(\R)$ (resp. an arbitrary $\ind_{r,s+2}$), cf. Cor.
\ref{p.FPCI.6A} and the proof of Theorem~\ref{p.FPCI.7}.  This will be
needed in particular for the formulation of the axioms which ensure
the uniqueness of the $\KO$--valued spectral flow, see Remark
\ref{p.ClkSF.20} 5.  and Sec.~\ref{ss.N} below.

\begin{prop}
Let $V$ be a finite dimensional $\Cl_{r,s+1}$--module with Clifford generators
$E_1,\ldots, E_r$, $F_1,\ldots,F_{s+1}$ representing a class in
$A_{r,s+2}$. Let $H_0$ be a standard $\Cl_{r,s+1}$ Hilbert space and 
$H:=H_0\oplus V$. Now put
\begin{align}
      J_0 := F_{s+1} = F_{s+1}\restr{H_0}\oplus F_{s+1}\restr{V}, \quad
      J_1 := F_{s+1}\restr{H_0}\oplus -F_{s+1}\restr{V}. 
\end{align}
Then $(J_0, J_1)$ is a Fredholm pair of elements of $\sJ^{r,s+1}(H)$
  with $\ind_{r,s+2}(J_0,J_1) = [V]$.
\end{prop}
\begin{proof} 
The difference $J_0-J_1$ is  finite-rank, hence $\|J_0-J_1\|_\sQ=0$
and $(J_0, J_1)$ is a Fredholm pair. Furthermore,
$\ker (J_0+J_1) = V$ with $\Cl_{r,s+1}$ structure given
by $E_1\restr V,\ldots, E_r\restr V$, $F_1\restr V,\ldots,
F_{s+1}\restr V$, hence $\ind_{r,s+2}(J_0,J_1) = [V]$.
\end{proof}

%#######################################
\section{The $\KO$--valued spectral flow}
\label{s.ClkSF}
%#######################################

This section is the core of the paper.
We give the definition of $\KO$--valued spectral flow 
 along with its basic properties.
We subsequently relate our definition to the approach of J. Phillips \cite{Phi96} and more generally 
show how our construction encompasses
previous instances of analytic spectral flow that have appeared in the literature.
Unless otherwise said, $H$ denotes a fixed $\Cl_{r,s}$--Hilbert space.

\subsection{Definition and properties}

We can now  define the 
$A_{r,s+2}\simeq  \sM_{r,s+1}/\sM_{r,s+2}\simeq \KO_{r,s+2}(\R)$--valued 
spectral flow of a path in $\sF^{r,s+1}$ with invertible
endpoints exactly as it was done in \cite{BLP05,L05}:

\begin{dfn}\label{p.ClkSF.3} Let $[0,1]\ni t\mapsto T_t\in\sF^{r,s+1}(H)$
be a continuous path of skew-adjoint Fredholm operators with
invertible end points $T_0, T_1$. Let $\Phi$ be the map defined 
in \Eqref{eq.NHCT.9}. Then
  \[
    t\mapsto \Phi(T_t) = 
       -F_{s} (I+ T_t F_{s})(I- T_t F_{s})\ii
  \]
is a path $\bl [0,1], \{0,1\}  \br \to \bl \tO_{r,s}, \tO_{r,s,*}\br$
in $\tO_{r,s}$ with endpoints in the  \emph{contractible} 
(see Lemma~\ref{p.NHCT.3}) neighbourhood
$\tO_{r,s,*}$ of $F_{s}$. Connect the endpoints of
 $\Phi(T_\bullet)$ arbitrarily, within $\tO_{r,s,*}$, to $F_{s}$. Then let 
\[
  \sff_{r,s+2}( T_\bullet) \in A_{r,s+2}\simeq \KO_{s+2-r}(\R)
\]  
be the class of the resulting loop in the stable fundamental
group $\pi_1^{\textup S}(\tO_{r,s}, F_{s})$. Since $\tO_{r,s,*}$ is
contractible, it is irrelevant how the path was closed up.
Alternatively, 
embed the Hilbert space $H$, equipped with the $\Cl_{r,s+1}$
structure with generators $E_1,\ldots,E_{r+s}, T_0|T_0|\ii$,
into a standard $\Cl_{r,s+1}$--Hilbert space $H'$ and take the class of the
loop obtained from $\Phi(T_t\oplus F_{s+1}\restr{H^\perp})$ in
the fundamental group $\pi_1(\tO_{r,s}(H'), F_{s})$.

  Finally, we put $\SF_{r,s+2}(T_\bullet):= \tau_{r,s+2}\bl
  \sff_{r,s+2}(T_\bullet) \br$ with $\tau_{r,s+2}$ as in Eq.~\eqref{eq.ABSC.16}.
\end{dfn}

\begin{remark}\label{p.ClkSF.20}
It follows from the Absorption Theorem mentioned in Sec.~\ref{ss.PN}
that the $\sff_{r,s+2}$ is well-defined. We collect some properties
of the spectral flow. These properties will be of relevance for the
discussion of the uniqueness in Sec.~\ref{s.USF} below.

\subsubsection*{1. Relation to the index of Fredholm pairs of complex structures}
Let $J_0, J_1\in\sJ^{r,s+1}(H)$ be a Fredholm pair of complex structures.
It then follows from Definition and Proposition~\ref{p.FPCI.5}
and Theorem~\ref{p.NHCT.4} that 
$\ind_{r,s+2}(J_0, J_1)$ equals the spectral flow $\sff_{r,s+2}$
of the straight line path $t\mapsto (1-t) J_0+ t J_1$.

\subsubsection*{2. Homotopy invariance}
If $T:[0,1]\times [0,1] \to  \sF^{r,s+1}(H)$ 
is a continuous map such that for all $u\in [0,1]$ the
endpoints $T(0,u)$, $T(1,u)$ are invertible operators, then
$\sff_{r,s+2}( T(\cdot,0 ) ) = \sff_{r,s+2}( T(\cdot,1) )$.

To see this, consider $\Phi(T(t,u))$. Now by Lemma~\ref{p.NHCT.3} and the proof of 
Prop.~\ref{p.NHCT.2} it follows that the endpoints 
$\Phi(T(0,u)), \Phi(T(1,u))$ can be connected to $F_{s}$
continuously in the parameter $u$ to obtain a homotopy of 
loops. Hence by definition 
$\sff_{r,s+2}( T(\cdot,0 ) ) = \sff_{r,s+2}( T(\cdot,1) )$.

\subsubsection*{3. Path additivity}
If $T^1_\bullet, T^2_\bullet:[0,1]\to \sF^{r,s+1}(H)$
are paths of Fredholm operators with invertible endpoints and
$T^1_1 = T^2_0$ then the $\sff_{r,s+2}$ of the concatenated path 
$T^1  * T^2$ equals the
$\sff_{r,s+2}(T^1_\bullet) + \sff_{r,s+2}(T^2_\bullet)$.

\subsubsection*{4. Stability} Let $T_\bullet:[0,1]\to \sF^{r,s+1}(H)$ be
a path of Fredholm operators with invertible endpoints. Let $H'$
  be another $\Cl_{r,s}$--Hilbert space and $S\in\sF^{r,s+1}(H')$ a
fixed invertible  operator. Then $\sff_{r,s+2}(T_\bullet)=
\sff_{r,s+2}(T_\bullet\oplus S)$.

\subsubsection*{5. Normalization} 

Let $V$ be a finite dimensional $\Cl_{r,s+1}$--module
(cf. Sec.~\ref{ss.SP})
and consider the straight line path
  $\gamma(t) = (1-2t)F_{s+1}$, $0\le t \le 1$ from $F_{s+1}$ to $-F_{s+1}$.
Then
\begin{equation}\label{eq.Norm}
  \sff_{r,s+2}(\gamma) = [V]\in A_{r,s+2} \simeq \sM_{r,s+1}/\sM_{r,s+2}.
\end{equation}

This follows immediately from the relation to the index of a pair
  since here $\sff_{r,s+2}(\gamma) = \ind_{r,s+2}(F_{s+1}, -F_{r,s+1})
  = [\ker (F_{r,s+1} + (-F_{r,s+1}) ) ] =[V]$.

\subsubsection*{6. Constancy}\label{sss.C}
The $\sff_{r,s+2}(T)$ of a constant
path (in which $T$ must be invertible) is $0$. 

\subsubsection*{7. Direct sum} \label{sss.DS}
If $T^1_\bullet:[0,1]\to \sF^{r,s+1}(H_1)$,
$T^2_\bullet:[0,1]\to \sF^{r,s+1}(H_2)$ are  paths of
skew-adjoint Fredholm operators with invertible endpoints,
then $\sff_{r,s+2}(T^1_\bullet) + \sff_{r,s+2}(T^2_\bullet)
= \sff_{r,s+2}(T^1_\bullet \oplus T^2_\bullet)$.
\end{remark}

\subsection{The `local formula' a la J. Phillips}

Next we will show that the spectral flow of a path $T_\bullet$
can be computed `locally' along the path by measuring the relative
indices of the phases of $T_\bullet$. This is the analogue of 
J. Phillips \cite{Phi96} approach in the complex case.

Let $[0,1]\ni t\mapsto T_t\in \sF^{r,s+1}(H)$ be a continuous path
with $T_0, T_1$ being invertible. The first difficulty is that
the path of phases $T_t |T_t|\ii$ is ill-defined as $T_t$ might not
be invertible. However, the invertibility of the endpoints
and the local constancy of the Fredholm index imply
$\ind_{r,s+1} T_t = 0 $ for all $0\le t\le 1$. Hence there exists a complex
structure $F_{s+1}\in\sJ^{r,s+1}(\ker T_t)$ turning $\ker T_t$ into a 
$\Cl_{r,s+1}$--module. We denote  by $J(T_t)\in\sJ^{r,s+1}(H)$  a
complex structure with
\begin{equation}\label{eq.ClkSF.4}
\begin{split}
       J(T_t)\big|_{\ker T_t^\perp} &= T_t |T_t|\ii,   \\
       J(T_t)\big|_{\ker T_t} & \in \sJ^{r,s+1}(\ker T_t )   \text{ arbitrary }.
\end{split}
\end{equation}
The complex structure $J(T_t)$ is unique only up to a finite rank perturbation, but
certainly $t\mapsto \pi\bl  J(T_t) \br\in\sQ(H)$ is continuous
in the Calkin algebra.

Alternatively, one might require, analogously to 
\cite[Sec.~4]{CPSchuba19}, that $J(T_t) = T_t |T_t|\ii$
only on the complement of a spectral subspace
$\ran\chi_{(0,a)}(|T_t|)$, where $0<a<\min\specess{|T_t|}$.
This would allow for example that $t\mapsto J(T_t)$ has
only finitely many discontinuities. The next Lemma says that
given a subdivision of the interval $[0,1]$, there is a finite-rank
perturbation $\tilde T_\bullet$ of $T_\bullet$ with the same
spectral flow and such that $T_t$ is invertible at the subdivision
points, moreover the phases may be prescribed.

\begin{lemma}\label{p.ClkSF.6}
  Let $T_t$ and $J_t:=J(T_t)$ be as before and let 
  $0=t_0<t_1\cdots <t_n=1$ be a subdivision
  of the interval $[0,1]$. Then there exists a continuous path
  $[0,1]\ni t\mapsto \tilde T_t\in\sF^{r,s+1}(H)$ with the following
  properties:
\begin{enumerate}
  \item $\tilde T_t - T_t$ is finite rank for all $t$.
  \item $\tilde T_{t_j}$ is invertible and 
    $\tilde T_{t_j} |\tilde T_{t_j}|\ii = J_{t_j}$,\; for \;$j=0,1,\ldots,n$.
  \item $\sff_{r,s+1}(\tilde T_\bullet )= \sff_{r,s+1}(T_\bullet)$.
\end{enumerate}
\end{lemma}
\begin{proof} For each $j=1,\ldots, n-1$ let, e.g.,
\[
    \begin{cases}
         R_j \big|_{\ker T_{t_j}} &:= J_{t_j}\big|_{\ker T_{t_j}}, \\
         R_j \big|_{\ker T_{t_j}^\perp} &:=0.
    \end{cases}
\]
Furthermore let $\varphi_j\in C^\infty[0,1]$ 
be smooth bump functions supported in a neighbourhood of $t_j$ resp.
such that their supports are disjoint and such that
$\varphi_j(t_j)=1$. Then $\tilde T_t:= T_t +
  \sum_{j=1}^{n-1} \varphi_j(t) R_j$ does the job.

  Property (3) follows since $\tilde T_\bullet$ and $T_\bullet$
have the same endpoints and hence the straight line
  $s\mapsto (1-s) \tilde T_\bullet + s T_\bullet$ is a homotopy
  between $\tilde T_\bullet$ and $T_\bullet$ keeping endpoints
  fixed.
\end{proof}  

\begin{theorem}[cf. {\cite[Prop.~6.3]{CPSchuba19}}] \label{p.ClkSF.7} 
Let $[0,1]\ni t\mapsto T_t\in \sF^{r,s+1}(H)$ be a continuous path
with $T_0, T_1$ being invertible. Let $J_t$ be complex structures
obtained by complementing the phase $T_t |T_t|\ii$ on 
  $\ker T_t^\perp$ by an arbitrary complex structure in $\sJ^{r,s+1}(\ker T_t)$,
cf. \Eqref{eq.ClkSF.4}.

Choose a partition $0=t_0<t_1<\cdots<t_n=1$ with the property that on
each segment $\| J_{t_{j-1}} - J_{t}\|_\sQ<1$, $t_{j-1}\le t\le t_j$.
Then
  \[
    \sff_{r,s+2}( T_\bullet ) =
       \sum_{j=1}^n \ind_{r,s+2}(J_{t_{j-1}}, J_{t_j} ).
\]
\end{theorem}
\begin{proof} We make repeated use of the homotopy invariance of
  $\sff_{r,s+2}$ and of $\ind_{r,s+2}$.

In view of Lemma~\ref{p.ClkSF.6} we may w.l.o.g. assume that
  $T_{t_j}$ is invertible for $j=0,\ldots, n$.  

Since $t\mapsto T_t$ is a continuous path of Fredholm operators
we have
\[
  r:= \min \bl \min_t \specess |T_t|, \min\spec |T_{t_0}|,\ldots, \min \spec
  |T_{t_n}|\br
    >0.
\]
We certainly have $J(T_t/r) = J(T_t)$ and by the homotopy invariance
  $\sff_{r,s+2}(T_\bullet) = \sff_{r,s+2}(T_\bullet/r)$. 

Thus without loss of generality it suffices to prove the Theorem for paths with
  $\min\specess |T_t| \ge 1$ and $\min\spec |T_{t_j}| \ge 1, j=0,\ldots,n$.

With $\Psi$ of \Eqref{eq.NHCT.12} consider the homotopy
  \begin{align*}
    H(u,t)&:= (1-u) T_t + u\Psi(T_t) = h_{u}(T_t), \\
    h_u(x)&:= (1-u) x + u \Psi(x).
  \end{align*}
For all $u$ the function $h_u$  is increasing and strictly
increasing in a neighbourhood of $0$. So for all $u$ we have
  $\ker H(u,t) = \ker T_u$ and $J(H(u,t)) = J(T_t)$.
Furthermore, by construction 
$h_1(T_{t_j}) = T_{t_j}|T_{t_j}|\ii =  J_{t_j}$.

Hence both sides of the claimed equation remain constant under
the deformation $H$. We are finally left to prove the claim for
the special case that $\specess T_t \subset \{\pm  i\}$ for all $t$
and $T_{t_j} = J_{t_j}$.

But then $T_t-J(T_t)$ is of finite rank and hence
$\|T_t-T_{t_{j-1}}\|_\sQ = \|J_{T_t}- J_{T_{j-1}}\|_\sQ<1$ for all
$t_{j-1}\le t\le t_j$.
Now $u\mapsto (1-u) ( (1-t) J_{t_{j-1}} + t J_{t_j}) + u T_t$
is a homotopy between $T_{t}, t_{j-1}\le t \le t_j$ and
the straight line path $(1-t) J_{t_{j-1}} + t J_{t_j}$.
The endpoints are kept fixed under the homotopy.
Consequently 
\[
\sff_{r,s+2}( (T_t)_{t_{j-1}\le t \le t_j} )
  =\sff_{r,s+2}(   (1-t) J_{t_{j-1}} + t J_{t_j}, 0\le t \le 1) 
  =\ind_{r,s+2}( J_{t_{j-1}}, J_{t_j} ).
\]
The path additivity
of the spectral flow then implies the claim.
\end{proof}  

The previous result allows us to relate the $\KO$--valued spectral
flow to the various known versions of spectral flow.

\begin{theorem} Let $T_\bullet\in\sF^{r,1}(H)$ be a family of
Fredholm operators with invertible endpoints. Let $K_1,K_2,L_1\in
  \Mat(2,\R)$ be the Clifford matrices \Eqref{eq.ICA.1}.
\subsubsection*{1. $r=2$ } 
Write $H= H'\otimes \R^2$ with $E_j=I_H\otimes K_j$, $j=1,2$.
Then $T_t=\tilde T_t\otimes K_1K_2$\footnote{Recall $K_1 K_2$ is the volume element of
$\Cl_{2,0}$.} with a family $\tilde T_\bullet$ of \emph{self-adjoint}
Fredholm operators and 
  $\SF(\tilde T_\bullet) =\SF_{2,2} (T_\bullet) \br$.

\subsubsection*{2. $r=0$. } In this case $T_\bullet$ is a family of 
skew-adjoint Fredholm operators without any additional symmetries.
Furthermore, $\SF_{0,2}\bl T_\bullet \br$ equals the
$\Z/2\Z$--valued Spectral Flow $\operatorname{Sf}_2$ of \cite{CPSchuba19}.

\subsubsection*{3. $r=1$. } Write again $H=H'\otimes \R^2$
with $E_1=K_1$. Then 
$T_t= \Re(\tilde T_t)\otimes K_2 + \Im(\tilde T_t)\otimes L_1$
with a family $\tilde T_\bullet$ of Fredholm operators on $H'$.
We have a forgetful map $\sF^{1,2}\subset \sF^{0,2}$ and 
 $\SF_{1,2}(T_\bullet) = \SF_{0,2}(T_\bullet) = \operatorname{Sf}_2(T_\bullet)$. 

Moreover, $\SF_{1,2}(T_\bullet)$ equals the parity
$\sigma(\tilde{T}_\bullet)$ of~\cite[Def.~3]{DSBW}.
More concretely, given a family $\tilde T_\bullet$
of Fredholm operators in a Hilbert space $H'$, put
\[
     T_t:= \Re(\tilde T_t)\otimes K_2 + \Im(\tilde T_t)\otimes L_1
         = \begin{pmatrix} 0 & -\tilde T^*_t\\ \tilde T_t & 0 \end{pmatrix}
\]
then
\[
    \sigma(\tilde T_\bullet) = \SF_{1,2}(T_\bullet) 
    = \SF_{0,2}( T_\bullet ) = \operatorname{Sf}_2(T_\bullet).
\]    
\end{theorem}  
\begin{proof} This follows immediately from the local formula
Theorem~\ref{p.ClkSF.7} and the corresponding statements
for the index of a pair in Sec.~\ref{s.FPCI}.
The claimed tensor product decompositions follow from
Prop.~\ref{p.ICA.1}, cf. also Remark~\ref{p.ICA.2} 2.
\end{proof}

\subsection{Examples from topological insulators}
\label{ss.TopPhaseExamples}

%{\color{blue}

Here we briefly outline some applications of our $\KO$--valued spectral flow to 
models in condensed matter physics. We first review results in~\cite{CPSchuba19} 
and building from this idea 
provide a (somewhat ad hoc) example, which by the normalization and 
homotopy invariance of 
spectral flow covers all remaining cases. We also outline
the relationship between paths of Hamiltonians in symmetry class AII with 
the quaternionic spectral flow $\sff_{r,r+4} \in KO_4(\R)$.

\begin{example}[Kitaev chain, cf. Sec.~10 of~\cite{CPSchuba19}] \label{ex:Kit_flow}
We consider a one dimensional Hamiltonian with particle-hole symmetry 
on $\ell^2(\Z)\otimes \C^2$, where 
$$
   \mathbf{H}(\mu,w) = \frac{1}{2}\begin{pmatrix} -w(S+S^*) - \mu & -i w(S^*-S)  \\ -i w(S^*-S) & w(S+S^*) + \mu \end{pmatrix}, 
   \qquad \mu, \, w \in \R
$$
and $S$ the unilateral shift operator on $\ell^2(\mathbb{Z})$. We see that 
$\mathbf{H}(\mu,w)$ anti-commutes with the particle-hole involution 
$C = \mathfrak{c}( I_{\ell^2(\Z)} \otimes K_2)$ with $\mathfrak{c}$ complex 
conjugation and $K_2$ from Eq.~\eqref{eq.ICA.1}.
For simplicity we will consider the case of $w=-1$, $\mu=0$, where the 
Hamiltonian is invertible. We furthermore condense 
our notation and write
$$
  \mathbf{H} = \mathbf{H}(0,-1) = \hat{S} + \hat{S}^*, 
  \qquad \hat{S} = S \otimes \frac{1}{2}\begin{pmatrix} 1 & i \\ i & -1 \end{pmatrix}.
$$
We now perturb this Hamiltonian by inserting a local flux through the unit cell at site $0$. 
Namely, we consider the path for $\alpha \in [0,1]$ 
$$
\mathbf{H}_{\alpha} = \hat{S}_\alpha + \hat{S}_\alpha^*, 
\quad \hat{S}_\alpha = S\otimes\frac{1}{2}\begin{pmatrix} 1 & i \\ i & -1 \end{pmatrix}
+ \nu_1 \nu_0^*\otimes \frac{1}{2} 
\begin{pmatrix} e^{-i\pi \alpha}-1 & i(e^{i\pi\alpha}-1) \\  i(e^{-i\pi\alpha}-1) & -(e^{i\pi\alpha}-1) \end{pmatrix},
$$
where $\nu_n$ is the partial isometry onto the site $n\in \mathbb{Z}$.

We see that $\mathbf{H}_{\alpha} - \mathbf{H}$ is finite-rank 
and therefore compact for any $\alpha$. We do not consider any additional symmetries 
or Clifford generators and so can compute $\SF_{0,2}(i\mathrm{H}_\bullet) = \SF_2(i\mathrm{H}_\bullet)$ 
using the $\Cl_{0,2}$ index of the (invertible) endpoints. 
Letting $J_k = i\mathbf{H}_k |\mathbf{H}_k|^{-1}$ for $k \in {0,1}$, 
it is shown in~\cite[Sec.~10]{CPSchuba19} that 
\[
  \SF_{0,2}(i\mathbf{H}_\bullet) = \ind_{0,2}(J_0, J_1) = 1.
\]
\end{example}

\begin{example}[$\KO$--valued spectral flow in the Kitaev chain]
Let us follow the same basic idea of Example~\ref{ex:Kit_flow} and consider more 
general $\KO$--valued spectral flow via a flux insertion of a one dimensional Hamiltonian. 
Let $V$ be a finite dimensional $\Cl_{r,s+1}$--module with Clifford generators
$E_1,\ldots, E_r$, $F_1,\ldots,F_{s+1}$ representing a class in $A_{r,s+2}$. 
We now consider the Hilbert space $\ell^2(\Z) \otimes V \otimes \C^2$ and 
again consider the spectral flow via a flux insertion through a unit cell. Namely, 
we 
consider the path of Hamiltonians 
$\mathbf{H}_\alpha = \hat{S}_\alpha + \hat{S}_\alpha^*$, where 
$$
\hat{S}_\alpha = S\otimes F_{s+1} \otimes \frac{1}{2}\begin{pmatrix} 1 & i \\ i & -1 \end{pmatrix}
+ \nu_1 \nu_0^*\otimes F_{s+1} \otimes  \frac{1}{2} 
\begin{pmatrix} e^{-i\pi \alpha}-1 & i(e^{i\pi\alpha}-1) \\  i(e^{-i\pi\alpha}-1) & -(e^{i\pi\alpha}-1) \end{pmatrix}.
$$
As in the case of the Kitaev chain, $\mathbf{H}_\alpha$ is a finite-rank perturbation 
of $\mathbf{H}_0$ that sends $F_{s+1}$ to $-F_{s+1}$ on the finite dimensional 
unit cell through which the flux is inserted. As such, we can easily compute that 
$$
  \sff_{r,s+2}( i\mathbf{H}_\bullet) = \ind_{r,s+2}(J_0, J_1) = [V] \in A_{r,s+2} \simeq \sM_{r,s+1}/\sM_{r,s+2}.
$$
Using this ad hoc construction, we can find a non-trivial $\KO$--valued 
spectral flow of any degree via a flux insertion in a Kitaev chain.
\end{example}

\begin{example}[$KO$--valued spectral flow in symmetry class AII]
Here we outline the basic framework for studying $KO$--valued spectral flow in systems 
with a fermionic time-reversal symmetry/quaternionic structure. As outlined in 
Sec.~\ref{ss.Alan.MPT}, our constructions are most easily illustrated using the framework 
of Kennedy and Zirnbauer~\cite{KZ16}. The relationship between the Kennedy--Zirnbauer 
framework and the more standard ten-fold way~\cite{RSFL} is outlined in~\cite{KZ16, AMZ19}.

We start with the Hilbert space $\mathcal{V}$ spanned by the fermionic creation operators 
(e.g. $\mathcal{V} = \ell^2(\Lambda)$ for some countable set $\Lambda$)
and form the Nambu space $\mathcal{W} = \mathcal{V} \oplus \mathcal{V}^*$ with 
real structure $C = \begin{pmatrix} 0 & R^{-1} \\ R & 0 \end{pmatrix}$ given by the 
Riesz map $R: \mathcal{V} \to \mathcal{V}^*$. We further assume that 
$\mathcal{V}$ has a quaternionic structure $T$, an anti-unitary map such that 
$T^2 = -I_\mathcal{V}$. Hamiltonians $\mathbf{H}$ on the Nambu space $\mathcal{W}$ are of 
Bogoliubov--de Gennes type, $\mathbf{H} = -C\mathbf{H}C$. 

The time-reversal and charge symmetry operators $\hat{T}$ and $Q$ are constructed on $\mathcal{W}$ 
as follows:
$$
  \hat{T} = \begin{pmatrix} T & 0 \\ 0 & RTR^{-1} \end{pmatrix}, \qquad 
  Q = \begin{pmatrix} I_\mathcal{V} & 0 \\ 0 & -I_{\mathcal{V}^*} \end{pmatrix}.
$$
We can use these operators to construct a $\Cl_{0,2}$ representation on the real 
Hilbert space $\mathcal{W}_\R$ of elements fixed by $C$. Namely, we take the generators 
$F_1 = C\hat{T}$, $F_2 = iC\hat{T}Q$ which commute with $C$.

Following~\cite{KZ16}, $\mathbf{H}$ is in symmetry class AII if 
$[\mathbf{H}, \hat{T}] = [\mathbf{H}, Q] = 0$. One then sees that 
$i\mathbf{H}$ anti-commutes with $F_1$ and $F_2$. Any continuous Fredholm path 
$\{\mathbf{H}_t\}_{t\in [0,1]}$ on $\mathcal{W}$ with invertible endpoints and 
such that $[\mathbf{H}_t,\hat{T}] = [\mathbf{H}_t, Q] = 0$ for all $t$ 
gives a path $\{i\mathbf{H}_t\}_{t\in[0,1]}$ in $\sF^{0,3}(\mathcal{W}_\R)$ with quaternionic spectral 
flow 
$$
\sff_{0,4}(i\mathbf{H}_\bullet) \in \sM_{0,3}/\sM_{0,4}, 
\qquad \SF_{0,4}(i\mathbf{H}_\bullet) \in \Z
$$
with $\SF_{0,4}$ measured by a sum of quaternionic Fredholm indices. 
%A concrete example of  a  path of Hamiltonians with 
%$\mathrm{SF}_{0,4}(i\mathbf{H}_\bullet) \neq 0$ can be written down using a flux insertion 
%or similar perturbations for Bogoliubov--de Gennes Hamiltonians on $\mathcal{W} = \ell^2(\Z^d) \otimes \C^{2m}$, 
%though such a computation is quite lengthy. 
%We instead provide this skeletal example as an illustration of the compatibility of our constructions with 
%the picture of free-fermionic topological phases provided by Kennedy and Zirnbauer.

We note that any Hamiltonian $\mathbf{H}$ such that $[\mathbf{H}, \hat{T}] = [\mathbf{H}, Q] = 0$ 
and $C \mathbf{H}C = -\mathbf{H}$ 
must be of the form $\mathbf{H} = \begin{pmatrix} h & 0 \\ 0 & -Rh R^{-1} \end{pmatrix}$ 
with $h$ a self-adjoint operator on $\mathcal{V}$ with $[h, T] = 0$. 
Considering $\mathcal{V}$ as a real vector space, any eigenvalue of $h$ will have 
a four-fold degeneracy coming from the quaternionic structure $\{i,j,k\} \sim \{i, T, iT\}$. 
Given a Fredholm path $h_\bullet$ in $\mathcal{V}$ with $h_t=h_t^*$, $[h_t, T] = 0$ for all $t$
and such that $h_0$ and $h_1$ are invertible, we can consider the standard spectral 
flow $\mathrm{SF}(h_\bullet)$.
 Passing to the Nambu space $\mathcal{V}\oplus \mathcal{V}^*$, because $\sigma(h_t) = \sigma(Rh_tR^{-1})$, 
 the kernel of $\mathbf{H}_t$ will be double the dimension of $\ker(h_t)$. 
This observation along with the doubling of Hilbert spaces gives us that 
$$
  \frac{1}{4} \SF(h_\bullet) =  \SF_{0,4}(i \mathbf{H}) \in \Z.
$$
As such, simple examples of non-trivial quaternionic spectral flow can be written down 
from a non-trivial and $T$-symmetric spectral flow of $h_\bullet$ in $\mathcal{V}$. 
We again remark that by taking complex dimension of the kernels, then 
this quaternionic spectral flow takes values in $2\Z$.
\end{example}

%}

%#######################################
\section{Extension to unbounded operators}
\label{s.EUO}
%#######################################

In applications the operators of interest are often differential operators and hence
unbounded. Our approach to $\sff_{r,s+2}$ via the Cayley transform
extends almost seamlessly to unbounded operators.  However, the
uniqueness proofs are slightly more complicated in the unbounded case
than in the bounded case, mainly due to complications in the homotopy
theory.  Our discussion parallels the complex case that is found in \cite{L05}.

%*********************************************************
\subsection{Topologies on spaces of unbounded operators}
\label{ss.TSUO}
%*********************************************************
We briefly introduce here the several topologies on unbounded
operators, for proofs and further details see \cite[Sec.~2]{L05}.
Let $H$ be a real Hilbert space. We recall that a densely defined 
operator $T:\dom(T)\subset H \to H$ is closed if $\dom(T)$ is a complete 
space with respect to the graph norm
\[
  \|v\|_{T} := \sqrt{\|v\|^2 + \|Tv\|^2} .
\]
We denote by $\sC(H)$ we denote the space of
densely defined closed operators in $H$ equipped with the so-called
gap topology. It can be described in several equivalent ways. For 
$T_1, T_2\in\sC(H)$ the gap distance is $d_G(T_1,T_2)=\|P_1-P_2\|$,
where $P_j$ denotes the orthogonal projection onto the graph of $T_j$
in the product space $H\times H$. For skew-adjoint $T_j$ the gap
distance is equivalent to $\| (T_1-I)\ii - (T_2-I)\ii\|$.

For skew-adjoint operators one can consider the \emph{Riesz transform}
$F(T):= T(I-T^2)\ii$. The corresponding \emph{Riesz metric}
$d_R(T_1,T_2)=\|F(T_1)-F(T_2)\|$ is stronger than the gap metric.

Finally, we consider operators with a fixed given domain: consider
a fixed skew-adjoint operator $D$ in $H$ with domain $W\hookrightarrow H$.
The operator $D$ serves as reference operator and is part of the
structure. Let
\[\sB^1(W,H)
   :=\bigsetdef{T\in \sC(H)}{ T=-T^* \text{ and } \dom(T)= W}.
\]   
This space is naturally equipped with the norm given by
$\|T\|_{W\to H}=\|T(I-D^2)^{-1/2}\|$. With $D$ being any fixed complex
structure on $H$ and $W=H$ we obtain $\sB^1(H,H)=\sB^1(H)$ the bounded
skew-adjoint operators on $H$ as a special case.

\begin{remark}\label{p.EUO.1} A word of warning: one might be tempted to
``identify'' $\sB^1(W,H)$ with the skew-adjoint operators in $H$ by
sending  
\[
  \sB^1(W,H)
   \ni T\to (I-D^2)^{-1/4} T (I-D^2)^{-1/4}=:B(T).
\]
Using a bit of interpolation theory (cf. \cite[Appendix]{L05}) one sees that
indeed $B(T)$ is a skew-adjoint bounded operator, and the map $T\mapsto B(T)$
is injective. However, it is in general \emph{not} surjective. Namely,
interpolation theory shows that for an operator $S\in\sB^1(H)$ to be in
the range of $B$ it is necessary that
\[(I-D^2)^{\alpha} S (I-D^2)^{-\alpha}\] is bounded on $H$ for
$-1/2\le \ga \le 1/2$.  In the interesting case of unbounded $D$, this
is a restriction which is not satisfied by all operators in
$\sB^1(H)$.
\end{remark}

If $H$ is a $\Clrs$--Hilbert space then we denote by $\sC^{r,s+1},
\sB^{r,s+1}, \ldots$ the corresponding spaces of skew-adjoint operators
which anti-commute with the Clifford generators $\Edots$, $\Fdots$,
cf. Sec.~\ref{ss.PN}. In the case of operators with a fixed domain
this requires, of course, that $D$ itself is $\Clrs$--anti-linear and
hence the inclusion $W\hookrightarrow H$ is an inclusion of
$\Clrs$--Hilbert spaces.  Similarly, $\sC\sF^{r,s+1}(H)$ and $\sF^{r,s+1}(W,H)$ denote
the corresponding spaces of skew-adjoint Fredholm operators.

The relations between the metric spaces of unbounded operators
are summarized by noting that that there
are natural continuous (inclusion) maps \cite[Prop.~2.2]{L05}
\begin{equation}\label{eq.EUO.1}
  \sB^{r,s+1}(W,H) \hookrightarrow (\sC^{r,s+1}(H),d_R) 
    \stackrel{\id}{\longrightarrow}
    (\sC^{r,s+1}(H), d_G),
\end{equation}
where $d_R, d_G$ denote the Riesz and gap distance resp.

%*********************************************************
\subsection{Spectral flow for families of unbounded operators}
\label{ss.SFFUO}
%*********************************************************

With the preparations of the previous section at hand, the 
definition of $\sff_{r,s+2}$ for families of unbounded skew-adjoint
Fredholm operators is straightforward. Let us
consider a $\Cl_{r,s}$--Hilbert space $H$ and a gap continuous
family $T_\bullet:[0,1]\to \sC\sF^{r,s+1}(H)$ of unbounded skew-adjoint
Fredholm operators with invertible endpoints. As remarked above,
the gap continuity implies (and is in fact equivalent to)
the continuity of the resolvent maps $t\mapsto (T_t\pm I)\ii$.
Hence, the Cayley transform (Definition~\ref{p.ClkSF.3})
\begin{equation}
    t\mapsto \Phi(T_t) = 
       -F_{s} (I+ T_t F_{s})(I- T_t F_{s})\ii
\end{equation}
is a continuous path $\bl [0,1], \{0,1\}  \br \to \bl \tO_{r,s},
\tO_{r,s,*}\br$
in $\tO_{r,s}$ with endpoints in the  \emph{contractible} 
(see Lemma~\ref{p.NHCT.3}) neighbourhood
$\tO_{r,s,*}$ of $F_{s}$. One is now in the situation of Def.
\ref{p.ClkSF.3} and therefore proceeding as in loc. cit. one obtains
an element $\sff_{r,s+2}( T_\bullet) \in A_{r,s+2}$
and $\SF_{r,s+2}:= \tau_{r,s+2}\bl  \sff_{r,s+2}( T_\bullet)  \br
\in \KO_{s+2-r}(\R)$, where $\tau$ was defined in \Eqref{eq.ABSC.16}.

The reader might notice that alternatively one could first have
introduced a spectral flow invariant of paths in 
$\bl [0,1], \{0,1\}  \br \to \bl \tO_{r,s}, \tO_{r,s,*}\br$ as the
basic object. The spectral flow of paths of bounded and unbounded
skew-adjoint Fredholm operators are then just obtained by pullback via
the map $\Phi$.

From the homotopy theory of $\tO_{r,s}$ it is clear that the
$\sff_{r,s+2}$ for gap continuous families of skew-adjoint Fredholm
operators satisfies \emph{Homotopy, Path additivity, Stability,
Normalization, Constancy}, and \emph{Direct sum} in the sense of
Remark~\ref{p.ClkSF.20}.

In view of the inclusions \Eqref{eq.EUO.1} the $\sff_{r,s+2}$ for Riesz
continuous families of skew-adjoint Fredholm operators
(resp.  such families in the space $\sF^{r,s+1}(W,H)$)
is just obtained by taking the $\sff_{r,s+2}$ after sending the family
to $\sC\sF^{r,s+1}$ via the natural maps \Eqref{eq.EUO.1}.

\subsection{Two technical results}
\label{ss.TTR}
In the unbounded case we will need adaptations of 
the two technical results \cite[Assertion 1 and 2 in Sec.~5]{L05}. 
The first statement is an improvement of the results in Lemma~\ref{p.ClkSF.6}.

\begin{lemma}\label{p.USF.3} Let $H$ be a $\Clrs$--Hilbert space 
and let $T_\bullet:[0,1] \ni t\mapsto \sC\sF^{r,s+1}(H)$ be a continuous
path with invertible endpoints. Then there exists a subdivision
$0=t_0<t_1\cdots <t_n=1$ of the interval $[0,1]$
and a continuous path
$\tilde T:[0,1]\to \sC\sF^{r,s+1}(H)$ with the following properties:
\begin{enumerate}
\item $\tilde T_t$ is a finite rank perturbation of $T_t$ and it has
  the same endpoints as $T_\bullet$, i.e.
$\tilde T_0=T_0, \tilde T_1=T_1$.
\item $\tilde T_{t_j}$ is invertible, $j=0,\ldots,n$.
The value of $T_{t_j}$ may even be prescribed as in Lemma
\ref{p.ClkSF.6}.
\item There exist $\eps_j>0$, $j=1,\ldots, n$ such that for all 
$t\in [t_{j-1},t_j]$ we have $\eps_j\not\in\spec |T_{t_j}| $ and
    $\specess(|T_{t_j}|)\cap [0,\eps_j]=\emptyset$.
\end{enumerate}

  The properties (1)--(2) imply that $\tilde T_\bullet$ is
(straight-line) homotopic to $T_\bullet$ and hence both paths
have the same $\sff_{r,s+1}$ spectral flow.

The result holds verbatim also for Riesz continuous paths in
  $\sC\sF^{r,s+1}(H)$ as well as for continuous paths into
  $\sF^{r,s+1}(W,H)$.
\end{lemma}
\begin{proof} To avoid repetitive statements a like 
`resp. for Riesz continuous resp for paths in $\sF^{r,s+1}(W,H)$'
we present the proof for gap continuous paths and leave the 
(obvious) modifications for the other two cases to the reader.
  
Clearly, since $T_t$ is Fredholm and by
compactness there is a subdivision $0=t_0< t_1<\ldots< t_n=1$
and $\eps_j>0$ such that $\eps_j\not\in\spec |T(t)| $ and
$\specess(|T(t)|)\cap [0,\eps_j]=\emptyset$. The problem
is that $T(t_j)$ need not be invertible. However,
an inspection of the proof of Lemma~\ref{p.ClkSF.6} shows, that
it holds for paths in $\sC\sF^{r,s+1}$ as well. More precisely,
the proof shows that there are finite rank $\Clrs$--antilinear
skew-adjoint operators $R_j$ (and the same choices as in
Lemma~\ref{p.ClkSF.6} are possible) with 
$R_j\restr{\ker T_{t_j}}\in \sJ^{r,s+1}(\ker T_{t_j})$
and $R_j\restr{\ker T_{t_j}^\perp} =0$.
With the notation of the proof of Lemma~\ref{p.ClkSF.6}
we see that for $\gl>0$ the operator family
$\tilde T_t(\gl):=T_t+\gl\cdot \sum_{j=1}^{n-1} \varphi_j(t) R_j$
  satisfies (1) and (2) for all $\gl>0$.
By compactness, there is a $\delta>0$  such that
for all $0< \gl\le \delta$ the family $\tilde T_\bullet(\gl)$
  satisfies (3) as well.    
\end{proof}

\begin{lemma}\label{p.USF.4} Let $H$ be a $\Clrs$--Hilbert space 
and let $T_\bullet:[0,1] \ni t\mapsto \sC\sF^{r,s+1}(H)$ be a gap continuous
path with invertible endpoints. Furthermore, assume that
\textup{(3)} of the previous Lemma is fullfilled for $T_\bullet$.
More precisely, assume that there exists an $\eps>0$ such that
for all $t\in [0,1]$ we have $\eps\not\in\spec |T_t| $ and
$\specess(|T_t|)\cap [0,\eps]=\emptyset$.

  Let $E(t) = 1_{[0,\eps]}(|T_t|)$ be the finite-rank
spectral projection onto the eigenspaces with eigenvalues
$\gl\le\eps$ of $|T_t|$.

Then there is a continuous family $U:[0,1]\to\sB(H)$ 
of \emph{unitary} operators commuting with the $\Clrs$--action
  such that $\tilde T_t = U(t)^* T_t U(t)$ has the following
  properties:
\begin{enumerate}
\item $\tilde T_\bullet:[0,1]\ni\to\sC\sF^{r,s+1}$
    with invertible endpoints.
\item $\tilde T_\bullet$ is homotopic to $T_\bullet$
    within paths in $\sC\sF^{r,s+1}$ with invertible endpoints.
  \item $\tilde T_t$ commutes with $E(0)$ and $\tilde T_t\restr{\ran
    E(0)}$ is invertible.
\end{enumerate}
  As a consequence $T_\bullet$ is homotopic within paths in
  $\sC\sF^{r,s+1}$ with invertible endpoints to a path
  $T^1_\bullet$ commuting with $E(0)$ and such that
  $T^1_\bullet\restr{\ker E(0)}$ is constant and invertible.

The result holds verbatim also for Riesz continuous paths in
  $\sC\sF^{r,s+1}(H)$ as well as for continuous paths into
  $\sF^{r,s+1}(W,H)$.
\end{lemma}
\begin{proof} Note that
\[
   E(t) = \frac{1}{2\pi i} \oint (- T_t^2 - z )\ii dz,
\]   
where the integral is along a contour encircling the interval
$[0,\eps^2]$ counterclockwise such that no other spectral points are
encircled. This is possible by the spectral gap assumptions. This
formula shows that $t\mapsto E(t)$ is continuous, in view of
\Eqref{eq.EUO.1} a fortiori if $T_\bullet$ is Riesz continuous or
fixed domain continuous.  Furthermore, $E(t)$ commutes with the
$\Clrs$--action.  The construction in the proof of
Prop.~\ref{p.NHCT.2} and in \cite[Prop.~4.3.3]{Bla98} shows that
there is a continuous family of $\Clrs$--linear unitaries
$U:[0,1]\to \sB(H)$, $U(0)=I$ such that $E(t) = U(t) E(0) U(t)^*$. Now
put $\tilde T_t:= U(t)^* T_t U(t)$. 

Since $E(t)$ and $U(t)$ are $\Clrs$--linear, the rest of the proof is
now identical to that in \cite[Assertion 2 in Sec.~5]{L05}.
\end{proof}

%*********************************************************
\subsection{The `local formula' for $\SF$ in the unbounded case and
consequences}
\label{ss.FINDLABEL}
%*********************************************************
Next we discuss the analogue of Theorem~\ref{p.ClkSF.7}.  For
this we will need the two technical Lemmas in Sec.
\ref{ss.TTR}. For Theorem~\ref{p.ClkSF.7} however, 
it is essential that the family $t\mapsto \pi(J_t)\in\sQ(H)$
depends continuously on $t$ where $J_t$ is essentially the phase
$T_t|T_t|\ii$.  For this to behave continuously, gap continuity is not
enough.  Therefore, the analogue of Theorem~\ref{p.ClkSF.7} can only
be proved for \emph{Riesz continuous} families of skew-adjoint Fredholm
operators.

\begin{theorem} \label{p.EUO.2} 
Let $H$ be a $\Cl_{r,s}$--Hilbert space and let 
 $[0,1]\ni t\mapsto T_t\in \sC\sF^{r,s+1}(H)$ be a Riesz continuous path with
$T_0, T_1$ being invertible. Let $J_t$ be complex structures obtained
by complementing the phase $T_t |T_t|\ii$ on $\ker T_t^\perp$ by an
  arbitrary complex structure in $\sJ^{r,s+1}(\ker T_t)$,
cf. \Eqref{eq.ClkSF.4}.

Then $t\mapsto \pi(J_t)\in\sQ(H)$ is continuous.
For a fine enough\footnote{The phrase `fine enough' can be quantified by
working through the constructions in the proof.} 
partition 
$0=t_0<t_1<\cdots<t_n=1$ of the interval $[0,1]$ we then
have the formula
\[
    \SF_{r,s+2}( T_\bullet ) =
       \sum_{j=1}^n \Ind_{r,s+2}(J_{t_{j-1}}, J_{t_j} ).
\]
\end{theorem}
\begin{proof} The continuity of $t\mapsto \pi(J_t)\in\sQ(H)$ follows
as in \cite[Lemma 3.3]{L05}. By the path additivity and homotopy
invariance of the spectral flow and by Lemma~\ref{p.USF.3},
it suffices to prove the Theorem for a family $T_t$ satisfying the
spectral gap assumptions of Lemma~\ref{p.USF.4}. We can further assume that 
$E(t)-E(0)$ and $U(t)-U(0)$ and the
Riesz distance $d_R(T_t,T_0)$ are small enough,
 and finally $T_0=J_0, T_1=J_1$.
So $T_\bullet$ now stands for the family on one of the segments
in Lemma~\ref{p.USF.3} and it remains to show that 
$\SF_{r,s+2}(T_\bullet)=\Ind_{r,s+2}(J_0,J_1)$.
With respect to the decomposition $H=\ker E(0)\oplus \ran E(0)$ we
have
  \[ U(t)^* T_t U(t) = S_t \oplus R_t\]
with an invertible family $S_\bullet$ and a finite rank family
  $R_\bullet$. 
Again by Homotopy and applying Theorem~\ref{p.ClkSF.7} and
Prop.~\ref{p.FPCI.8} to the finite-rank family $R_\bullet$
we have
\[
   \begin{split}
     \SF_{r,s+2}( T_\bullet ) 
        & =  \SF_{r,s+2}\bl U(\bullet)^* T_\bullet U(\bullet) \br\\
        & =  \SF_{r,s+2} \bl R_\bullet \br \\
        & =  \Ind_{r,s+2} \bl 
           U(0)^* J_0\restr{\ran E(0)} U(0), 
             U(1)^* J_1 U(1)\restr{\ran E(0) } \br\\
        & =  \Ind_{r,s+2} \bl U(0)^* J_0 U(0), U(1)^* J_1 U(1) \br\\
        & =  \Ind_{r,s+2} \bl  J_0, U(0)U(1)^* J_1 U(1)U(0)^* \br \\
        & = \Ind_{r,s+2}  (J_0, J_1).
   \end{split}
\]
From line 3 to line 4 it was used that $S_0, S_1$ is an invertible
pair since they are close enough in the Riesz distance. For the last 
equality,  one may invoke the homotopy
  $U(0)U(s)^* J_1 U(s) U(0)^*$, $0\le s\le 1$, again using
   the smallness assumptions which imply that this is a valid
  homotopy of Fredholm pairs.
\end{proof}  

\begin{theorem}[cf. {\cite[Theorem 3.6]{L05}}]     \label{p.EUO.3} 
Let $H$ be a $\Cl_{r,s}$--Hilbert space and let 
  $[0,1]\ni t\mapsto T_t\in \sC\sF^{r,s+1}(H)$ be a Riesz continuous path with
$T_0, T_1$ being invertible. Assume furthermore, that the domain of
$T_t$ does not depend on $t$ and that for each $t\in[0,1]$ the
difference $T_t-T_0$ is relatively $T_0$-compact, i.e. 
$(T_t-T_0)(T_0 +I)^{-1}$ is compact.
  Then the pair 
$(T_0 |T_0|\ii, T_1|T_1|\ii)$ consisting of the phases of the
  endpoints is a Fredholm pair of complex structures in $\sJ^{r,s+1}(H)$ and
we have the index formula
\[
    \sff_{r,s+2} ( T_\bullet ) = \ind_{r,s+2} ( T_0|T_0|\ii, T_1|T_1|\ii ).
\]  
\end{theorem}
\begin{proof} It follows from \cite[Prop.~3.4]{L05} that under the
assumptions of the Theorem the differences of the Riesz transforms
$F(T_t)-F(T_0)$ is a compact operator for all $t\in [0,1]$. Similarly,
the differences of the `phases' $J_t-J_0$ are compact as well.
Now the claim follows by applying Prop.~\ref{p.FPCI.8} to the
formula in Theorem~\ref{p.EUO.2}.
\end{proof}

Thus, for a path of relatively compact operators the $\SF_{r,s+2}$
depends only on the endpoints. This applies in particular to
the finite dimensional case which we single out here.

\begin{cor}\label{p.EUO.4} 
Let $H$ be a finite dimensional $\Cl_{r,s}$--module and let 
  $[0,1]\ni t\mapsto T_t\in\sB^{r,s+1}(H)$ be a continuous family of
skew-adjoint matrices anti-commuting with the Clifford generators. Then
$\sff_{r,s+2} ( T_\bullet ) = \ind_{s,r+2} ( T_0|T_0|\ii, T_1|T_1|\ii )$.  
\end{cor}

%###########################################################
\section{Uniqueness of the $\KO$--valued spectral flow}
\label{s.USF}
%###########################################################

In this section we give an `axiomatic' characterization of $\KO$--valued spectral flow.
Our approach is not exactly the same as in the complex case as described in \cite{L05}
because we have to start with the finite dimensional case. As a result, stability must  play a role in
the passage to infinite dimensions.

%*********************************************************
\subsection{Uniqueness in the finite dimensional case}
\label{ss.UFDC}
%*********************************************************

We introduce some notation. For a finite dimensional
$\Cl_{r,s}$--module $H$ we denote by $\sA^{r,s+1}(H)$ the set of
continuous paths $\gamma:[0,1]\to \sB^{r,s+1}(H)$ with $\gamma(0),
\gamma(1)$ being invertible. 

\begin{theorem}\label{p.USF.1} Suppose we are given,
for each finite dimensional $\Cl_{r,s}$--module $V$,
a map $\mu_V:\sA^{r,s+1}(V) \to A_{r,s+2}$ satisfying \emph{Homotopy,
Path additivity, Stability}, and \emph{Normalization} in the sense
of Remark~\ref{p.ClkSF.20}.  Then $\mu_V = \sff_{r,s+2}$.
\end{theorem}
More precisely, Normalization means the following:
if $V$ is a module as in \Eqref{eq.Norm} (resp. Sec.~\ref{ss.SP})
then $\mu_V(t\mapsto (1-2t) F_{s+1}) = [V]$.

\begin{proof} Suppose first, we are given two paths 
$\gamma_1, \gamma_2\in\sA^{r,s+1}(V)$ with the same endpoints.
  Then the straight line homotopy $u\mapsto (1-u) \gamma_1 + u \gamma_2$
stays within $\sA^{r,s+1}(V)$, hence
$\mu_V(\gamma_1)=\mu_V(\gamma_2)$. Thus for
$\gamma\in\sA^{r,s+1}(V)$ the value $\mu_V(\gamma)$ depends only on
the endpoints of $\gamma$.  In particular, 
$\mu_V(\gamma)=\mu_V(t\mapsto (1-t) \gamma(0)+t \gamma(1))$.

If $h_0(u), h_1(u)$ are paths of invertible elements in
$\sB^{r,s+1}(V)$ with $h_0(0)=\gamma(0)$ and $h_1=\gamma(1)$,
then $H(u,t):=(1-t) h_0(u) + t h_1(u)$ is a homotopy with invertible
endpoints from the straight line path from $\gamma(0)$ to $\gamma(1)$
to the straight line path from $h_0(1)$ to $h_1(1)$.

Together with Stability this shows that $\mu_V(\gamma)$ depends
only on the stable homotopy classes of the endpoints of the paths.
Since any invertible $T\in\sB^{r,s+1}(V)$ is homotopic to its phase
$T|T|\ii\in\sJ^{r,s+1}(V)$, it suffices to consider pairs of
complex structures. By Appendix~\ref{ss.SHFC} we know that
in each stable path component there is a standard pair as 
in Sec.~\ref{ss.SP}. Hence by Normalization 
(and Stability) the claim follows.
\end{proof}  

%*********************************************************
\subsection{Uniqueness in the fixed domain unbounded case 
(includes bounded case)}
\label{ss.UFD}
%*********************************************************

Here we discuss uniqueness for families in $\sF^{r,s+1}(W,H)$. This
contains the case $W=H$ and hence the bounded case as a special case.
The obvious extension to gap or Riesz continuous families is left to
the reader, but see \cite[Sec.~5]{L05}.

So fix a $\Cl_{r,s}$--Hilbert space and a skew-adjoint operator
$D\in\sC\sF^{r,s+1}(H)$ with domain $W$. By $\sA^{r,s+1}(W,H)$ we denote
the set of continuous paths $\gamma:[0,1]\to \sF^{r,s+1}(W,H)$ with
invertible endpoints. A pair $(\tilde W,\tilde H)$ is called
$W$--admissible if there is a finite dimensional $\Cl_{r,s+1}$--module $V$
such that $\tilde W=W\oplus V$, $\tilde H=H\oplus V$ as
$\Cl_{r,s}$--Hilbert spaces.  Hence $\tilde W\hookrightarrow \tilde H$
sits naturally in $\tilde H$ as a dense subspace. $\tilde W$ is the
domain, e.g., of the skew-adjoint operator $D\oplus F_{s+1}\restr{V}$.

\begin{theorem}\label{p.USF.2}  Suppose that for each $W$-admissible
pair $(\tilde W, \tilde H)$ we are given a map 
$\mu_{\tilde W}: \sA^{r,s+1}(\tilde W,\tilde H)\to A_{r,s+2}$ 
satisfying \emph{Homotopy, Path additivity}, and \emph{Stability}
in the sense of Remark~\ref{p.ClkSF.20}. Furthermore, assume that
$\mu$ satisfies the following variant of \emph{Normalization}:
\begin{quote} 
Let $T\in\sF^{r,s+1}(W,H)$ be an invertible operator. Furthermore,
let $V$ be a finite dimensional $\Cl_{r,s+1}$--module. Consider the path
  \[
    \gamma(t):= T \oplus (1-2t) F_{s+1}\in \sF^{r,s+1}(W\oplus
    V,H\oplus V).
  \]  
    Then $\mu_{W\oplus V}(\gamma) = [V]\in A_{r,s+2}$.
\end{quote}
Then $\mu_{\tilde W}=\SF_{r,s+2}$ for all admissible pairs 
$(\tilde W,\tilde H)$.
\end{theorem}  
\begin{proof} Let $T_\bullet$ be given. Apply Lemma~\ref{p.USF.3}
and let $\tilde T_\bullet$ be the corresponding finite-rank
modification of $T_\bullet$. By Homotopy and Path additivity
we have
\[
      \mu_W(T)   =\mu_W(\tilde T) 
            = \sum_{j=1}^n \mu_W\bl  \tilde T\restr{[t_{j-1}, t_j]}
            \br.
\]
We look at one of the $\tilde T\restr{[t_{j-1}, t_j]}$
and write again $T$ for it and parametrize it over $[0,1]$.
As a result of this exercise we may now assume that 
there is a global spectral gap $\eps>0$ such that
$T$ satisfies the prerequisites of Lemma~\ref{p.USF.4}.
By homotopy
\[
     \mu_W( T_\bullet ) = \mu_W( R_t\oplus S),
\]
where $S$ is constant and invertible on $\ker E(0)$ and
$R_t$ is finite-rank on $\ran E(0)$. 

There is a slight difficulty here, as $\ker E(0)$ is
only a subspace of $H$ (resp. $W$). Since $R_0$ is also
  invertible on $\ran E(0)$ we put $\tilde W:= \ran E(0)\oplus W$.
Then by Stability (and Homotopy to switch the roles of the two
  copies of $\ran E(0)$)
\[
  \mu_W(T_\bullet) = \mu_{\tilde W}( R_t \oplus (R_0\oplus S) ).
\]
By Normalization and the  finite dimensional case, which has been proved, 
  the latter equals $\sff_{r,s+2}( T_\bullet)$ and we are done.
\end{proof}

\section{The Robbin--Salamon Theorem for $\KO$--valued spectral flow}
\label{s.RS}

The axiomatic characterization of the $\KO$--valued spectral flow in the previous section
allows a short proof of our main objective. Namely we now prove the `spectral flow $=$ Fredholm index' result
alluded to in the introduction.

\subsection{Standing assumptions}\label{ss.SA} Following
\cite[(A1)--(A3)]{RobSal:SFM}, \cite[Sec.~8]{KaaLes:SFU} 
we need to introduce quite a bit of notation here. We continue
to work in the framework of Fredholm operators with a fixed
domain. So let $H$ be a $\Clrs$--Hilbert space
$H$ and let $D\in\sC\sF^{r,s+1}(H)$ be a skew-adjoint
Fredholm operator \emph{with compact resolvent}. The latter
is equivalent to the saying that the domain $W:=\dom(D)$ with
its natural graph norm is compactly embedded into $H$. 

Furthermore, consider a one-parameter 
family $A(\cdot):\R \mapsto \sF^{r,s+1}(W,H)$
of skew-adjoint Fredholm operators such that the following
conditions are satisfied:
\begin{enumerate}\renewcommand{\labelenumi}{\textup{(A \arabic{enumi})}}
\item\label{Ass1} 
  The map $A : \R \to \sB^{r,s+1}(W,H)$ is weakly differentiable. 
This means that for all $\xi\in W$ and all $\eta\in H$ the map
$t \mapsto \inn{A_1(t)\xi,\eta}$ is differentiable.
Furthermore, we suppose that the weak derivative $A'(t) : W \to H$
is bounded for each $t \in\R$ and that the supremum
    $\sup_{t \in\R}\|A'(t)\|_{W\to H} =: K < \infty$ is finite.

  \item\label{Ass2} The domain $\dom(A(t))=W$ is independent of $t$
and equals $W$. Moreover, there exist constants  $C_1,C_2 > 0$ such that
\begin{equation}\label{eq:Ass2}
C_1 \|\xi\|_{W} \leq \|\xi\|_{A(t)} \leq C_2 \|\xi\|_{W}
\end{equation}
for all $\xi \in W$ and all $t \in\R$. In other words, the
graph norms are uniformly equivalent to the norm $\|\cdot\|_W$ of $W$.

\item\label{Ass3} There exists $R,c>0$ such that for all $t\in\R$ with
  $|t|\ge R$ we have $|A(t)|\ge c\cdot I$, i.e. $A(t)$ is
invertible and uniformly bounded below.
\end{enumerate} 
\newcommand{\Aref}[1]{\textup{(A \ref{#1})}}
\newcommand{\AssA}{\Aref{Ass1}}
\newcommand{\AssB}{\Aref{Ass2}}
\newcommand{\AssC}{\Aref{Ass3}}
\AssA\ and \AssB\ are the same as in \cite{RobSal:SFM}, while \AssC\ 
is slightly more general than \cite[(A3)]{RobSal:SFM}, see
\cite[Sec.~8]{KaaLes:SFU}.

If $A(\cdot)$ satisfies \AssA--\AssC\ then we let $\sff_{r,s+2}(A(\cdot))$
be the $\sff_{r,s+2}$ of the path $[-R,R]\ni t\mapsto A(t)$.
We leave the straightforward verification that this is well-defined
to the reader. The condition \AssA\ implies continuity as 
a map $\R \to \sB^{r,s+1}(W,H)$. Furthermore, \AssA\ can
be relaxed to the assumption that $A(\cdot)$ is continuous and
piecewise weakly continuously differentiable. This has the
benefit that given $A(t)_{a\le t\le b}$ satisfying the 
finite interval analogues of \AssA--\AssC\ with $A(a), \, A(b)$ invertible, the
family 
\[
  \tilde A(t):= 
  \begin{cases}
    A(a), & t\le a,\\
    A(t), & a\le t\le b, \\
    A(b), & t \ge b,
   \end{cases}
\]
satisfies the modified axiom \AssA\ and the axioms \AssB, \AssC.

\subsection{The Robbin--Salamon Theorem for $\Clrs$--antilinear
skew-adjoint Fredholm operators}
After these preparations, we are ready to formulate the main
result of this section:

\begin{theorem}
\label{thm:Robbin-Salamon}
Let $A(\cdot)$ be a family of unbounded skew-adjoint Fredholm operators
  in $\sC\sF^{r,s+1}(W,H)$ satisfying the axioms \AssA--\AssC\ above.
Let 
\begin{align}
  D  & = A(t) \otimes \go_{1,1} - \frac{d}{dt}\otimes K_1\\
     & =\begin{pmatrix} -\frac d{dt} & -A(t) \\ - A(t) & \frac d{dt}\end{pmatrix}.
\end{align}
Then
 $D$ is an essentially skew-adjoint Fredholm operator in
$\sC\sF^{r,s+2}( L^2(\R, H\otimes \R^2) )$
where the Clifford generators are given by, cf. Prop.~\ref{p.ICA.1}, 
  $\tilde E_j:= E_j\otimes \go_{1,1}$, $j=1,\ldots,r$,
  $\tilde F_k:= F_k\otimes \go_{1,1}$, $j=1,\ldots,s$,
  $\tilde F_{s+1} = - I\otimes L_1$.
With these normalizations we have
\[
  \sff_{r,s+2}( A(\cdot) ) = \ind_{r,s+2}(D).
\]
\end{theorem}
\begin{remark} \label{rk:RS_Clifford_normalization}
There are other possible normalizations of course. 
For example, the roles of the $2\times 2$ matrices $K_1, K_2$ can (up to
sign) be reversed.
The self-adjoint unitary 
$\frac 1{\sqrt 2}(K_1-K_2)$ conjugates $K_1$ and $-K_2$. It is then
not hard to see that $D$ can be replaced by
\begin{equation}
    D' = \begin{pmatrix} A(t) & \frac d{dt} \\ \frac d{dt}  & -A(t) 
    \end{pmatrix} = A(t)\otimes K_1 + \frac d{dt}\otimes K_2
\end{equation}
with Clifford generators given by
  $\tilde E_j:= E_j\otimes K_1$, $j=1,\ldots,r$,
  $\tilde F_k:= F_k\otimes K_1$, $j=1,\ldots,s$,
  $\tilde F_{s+1} = I\otimes L_1$.
\end{remark}

\begin{proof} For the analytic part of the theorem, the
results of \cite{RobSal:SFM} and \cite{KaaLes:SFU} apply.
Note that we may write
\[
  D= -(1\otimes K_1) \bl \frac d{dt} - A(t)\otimes L_1\br.
\]
Thus viewing $H\otimes \R^2$ as a complex Hilbert space
with complex structure $-I\otimes L_1$ we see that 
$\tilde A(t):=-A(t)\otimes L_1$ is a family of self-adjoint
Fredholm operators satisfying \AssA--\AssC.

Consequently, $D$ is essentially skew-adjoint and Fredholm.
Furthermore, a direct  calculation shows that $D$ is
$\Cl_{r,s+1}$--anti-linear with respect to the Clifford
matrices $\tilde E_1,\ldots,\tilde E_r$,
$\tilde F_1,\ldots,\tilde F_{s+1}$ given in the theorem.

It is clear that the right hand side of the claimed index formula
satisfies homotopy invariance. Concatenation can either be proved
directly by invoking classical methods of elliptic boundary value
theory or as in \cite[Prop.~4.26]{RobSal:SFM} where it is shown that
concatenation already follows from homotopy and the easy to check
constant and direct sums axioms.

It remains to check normalization which is done below.
\end{proof}

%*********************************************************
\subsection{Normalization}
\label{ss.N}
%*********************************************************

Let $V$ be a finite dimensional $\Cl_{r,s+1}$--module and consider the path
\[
  f(t) F_{s+1} =: A(t),
\]
where $f:\R\to \R$ is a continuous function with
\[f(t)= \begin{cases} 1, & t\le 0\\ -1, & t\ge 1.
      \end{cases}
\]

The path $A(t)$ is homotopic to the straight line path
from $F_{s+1}$ to $-F_{s+1}$ in the finite dimensional 
$\Cl_{r,s+1}$--module $V$, hence by Sec.~\ref{ss.SP} we have
\[
  \sff_{r,s+2}( A(\cdot) ) = [V] \in A_{r,s+2}.
\]
We will need the $L^2$-solutions of the ODE
\begin{equation}
    u' + \eps f u =0, \quad \eps\in\{\pm 1\}
\end{equation}
on $\R$. For $x\le 0$ we have $u(x) = e^{-\eps x} u(0)$
and for $x\ge 1$ we have $u(x) = e^{\eps x} u(1)$.
In both cases the solution is square integrable if and only if $\eps = - 1$.
Thus, for $\eps=1$ there are no $L^2$-solutions and for
$\eps=-1$ there is a unique $L^2$-solution $\bar u$ with
$\bar u(0)=1$.

With regard to the consideration about $L^2$-solutions we
infer that the map
\[
    \Phi: V\mapsto \ker D, \quad \eta\mapsto \Bl x\mapsto 
      \frac{\bar u(x)}{\sqrt 2}
      \begin{pmatrix} \eta-F_{s+1} \eta \\ \eta+ F_{s+1}\eta
      \end{pmatrix}\Br
\]
is an isomorphism. One immediately checks, that $\Phi$
is $\Cl_{r,s+1}$--equivariant in the sense that
$\Phi( E_j \eta ) = \tilde E_j \Phi(\eta)$, $j=1,\ldots,r$
and $\Phi( F_k\eta) = \tilde F_k\Phi(\eta)$, $k=1,\ldots,s+1$.
Thus $\Phi$ is an isomorphism of the Clifford module
$(V; E_1,\ldots,E_r, F_1,\ldots, F_{s+1})$ onto the Clifford module
$(\ker D;\tilde E_1,\ldots,\tilde E_r, \tilde F_1,\ldots,\tilde
F_{s+1})$ and thus
$\ind_{r,s+2} D = [V] =\sff_{r,s+2}( A(\cdot) )$.
The proof of Theorem~\ref{thm:Robbin-Salamon} is complete.
\hfill\qed

\section{Spectral flow and the Kasparov product}
\label{s.SFKP}

In this section we relate  Theorem~\ref{thm:Robbin-Salamon} to
Kasparov's theory via the unbounded Kasparov product. In order to have
both a spectral flow and a Robbin--Salamon interpretation of the index
theorems provided by Kasparov's theory, we must be in a setting where
we have a one-parameter path of operators. 
For analogous results in the complex case, see~\cite{AzzaliWahl, vdDungen, KaaLes:SFU}.

We first note a word of caution in that 
up until this point, we have largely ignored the gradings on Clifford algebras. 
However, $\KK$-theory is intrinsically $\Z/2\Z$--graded and as such all Clifford 
representations that appear as part of a Kasparov module 
must respect the grading. We will explicitly state when we can consider 
representations and modules as ungraded.

Throughout this section we use the identification of $\KU$-theory with
$\KK$-theory contained in the next Lemma.

\begin{lemma}
\label{lem:PD-iso}
For $\sA$ a $\sigma$-unital $C^*$-algebra, we have for $s\geq r$ 
\[
\KO_{s-r}(\sA):=\KKO(\R,\sA\ox C_0(\R^{s-r})) \simeq
\KKO(\Cl_{s,r},\sA)
\]
where 
 the isomorphism is given by the Kasparov product with the class
  of the ($\Z/2\Z$--graded) spectral triple
\[
\lambda_{s-r}=\Big( C_0(\R^{s-r})  \ox \Cl_{s-r,0} , \, L^2( \R^{s-r}, \Lambda^* \R^{s-r}), \, d + d^* \Big)
\]
with $d+d^* = \sum_{j=1}^{s-r} \partial_{x_j} \ox f_j$ using 
the isomorphism $\sB(\Lambda^* \R^k) \simeq \Cl_{k,k}$. 
\end{lemma}
\begin{proof}
This is a special case of \cite[Theorem 7, Sec.~5]{Kas80}.
\end{proof}

Let us first review the equivalence between paths of operators in 
$\sFsrs$ and Kasparov modules. 

\begin{prop} \label{prop:skew_Fredholm_to_Kasmod}
Let $\{T(t)\}_{t\in\R}$ be a norm-continuous path in $\sF_{\ast}^{r,s+1}$ with 
$T(t)$ invertible for all $t$ outside the compact set $[-R,R]$. 
Then 
\begin{equation} \label{eq:aah_another_kas_mod}
  \left( \Cl_{s+1,r}, \, \begin{pmatrix} H \ox C_0(\R) \\ H \ox C_0(\R) \end{pmatrix}_{C_0(\R)} , \, 
   \begin{pmatrix} 0 & -T(\cdot) \\ T(\cdot) & 0 \end{pmatrix} \right)
\end{equation}
is a real Kasparov module 
where the left $Cl_{s+1,r}$--representation is generated by the elements 
 $\left\{ I\ox K_2, F_1 \ox L_1, \ldots, F_s \ox L_1, E_1 \ox L_1, \ldots, E_r \ox L_1 \right\}$.
\end{prop}
\begin{proof}
Because $\{T(t)\}_{t\in\R}$ is a norm-continuous path of Fredholm operators 
invertible outside a compact set, it follows that $T(\cdot)$ is a well-defined 
operator on the $C^*$-module $H \ox C_0(\R)$. 
It is then a simple check that the matrix $T(\cdot) \ox L_1$ is self-adjoint and 
anti-commutes (graded-commutes) with the left Clifford action. 
\end{proof}

\begin{remark}
Analogously to the previous proposition, if $A(\cdot):\R \mapsto \sF^{r,s+1}(W,H)$ is a 
one-parameter family of skew-adjoint operators with compact resolvent and satisfying  \AssA--\AssC\ as in 
the Robbin--Salamon theorem. Then we can construct an \emph{unbounded} 
Kasparov module 
\begin{equation} \label{eq:RS_Kasmod}
  \left( \Cl_{s+1,r}, \, \begin{pmatrix} H \ox C_0(\R) \\ H \ox C_0(\R) \end{pmatrix}_{C_0(\R)} , \, 
   \begin{pmatrix} 0 & -A(\cdot) \\ A(\cdot) & 0 \end{pmatrix} \right)
\end{equation}
with Clifford generators as in Prop.~\ref{prop:skew_Fredholm_to_Kasmod}. 
\end{remark}

The properties of the spectral flow listed in Remark~\ref{p.ClkSF.20} 
 ensure that $\sff_{r,s+2}$ descends to a map
\[
\sff_{r,s+2}:\,\KO_{s+1-r}(C_0(\R))\to \KO_{s+2-r}(\R).
\]
We will use Theorem~\ref{thm:Robbin-Salamon}
to compare spectral flow of the family 
$A(\cdot):\R \mapsto \sF^{r,s+1}(W,H)$ satisfying \AssA--\AssC\ 
with the Kasparov product of the Kasparov module represented by 
Equation \eqref{eq:RS_Kasmod} with the fundamental class for $\R$.
The fundamental class $[\partial_x]\in \KKO(C_0(\R)\ox\Cl_{1,0},\R)$ 
of the real line is  represented
by the unbounded Kasparov module
\begin{equation}
\left( C_0(\R)\ox\Cl_{1,0}, \, L^2(\R) \otimes \Lambda^* \R , \, \partial_x\ox f \right),
\label{eq:fun-class}
\end{equation}
where $\Lambda^* \R \simeq \R^2$, $f=\begin{pmatrix} 0 & -1\\ 1 & 0\end{pmatrix}$ 
and the representation of $\Cl_{1,0}$ is generated by
$e=\begin{pmatrix} 0 & 1\\ 1 & 0\end{pmatrix}$.

\begin{prop}
\label{prop:kasprod=specflow}
Let $A(\cdot):\R \mapsto \sF^{r,s+1}(W,H)$ be a 
one-parameter family of skew-adjoint operators with compact resolvent and 
satisfying  {\AssA--\AssC}. 
The Kasparov product $[A(\cdot)]\ox_{C_0(\R)}[\partial_x]$ 
of the classes of the Kasparov modules 
\eqref{eq:RS_Kasmod}
and \eqref{eq:fun-class} is represented by
the Kasparov module
\[
\left( \Cl_{s+2,r}, \, L^2(\R, H) \ox \Lambda^* \R^2, \, A(\cdot) \ox f_1 + \partial_x \ox f_2  \right)
\]
where we identify $\sB(\Lambda^* \R^2) \simeq \Cl_{2,2}$ and the Clifford generators of 
the left action are 
$\{I \otimes e_1, I\otimes e_2, F_1 \otimes f_1, \ldots, F_s\otimes f_1, E_1\otimes f_1, \ldots, E_r\otimes f_1 \}$.

Applying the isomorphism 
$\phi:\,\KKO(\Cl_{s+2,r},\R)\to \KO_{s+2-r}(\R)$
to the product yields
\[
\phi \big([A(\cdot)]\ox_{C_0(\R)}[\partial_x] \big) 
 = \ind_{r, s+2} \begin{pmatrix} A(\cdot) & \partial_x\\ \partial_x & -A(\cdot) \end{pmatrix} 
 =   -\sff_{r,s+2}(A(\cdot)).
\]
\end{prop}
\begin{proof}
To take the (unbounded) Kasparov product, we first take 
$d:C^1_0(\R)\to \sB(L^2(\R))$ the trivial connection. Then making the identification 
\[
C_0(\R, H)\ox_{C_0(\R)} L^2(\R)\to L^2(\R, H),
\]
we have that $1\ox_d\partial_x$ is represented as  $\partial_x$ on 
$L^2(\R, H)$. We can similarly identify $\Lambda^* \R \otimes \Lambda^* \R \simeq \Lambda^* \R^2$ 
and so our product Hilbert space is $L^2(\R, H) \ox \Lambda^* \R^2$. 
Putting together our Clifford actions, the results in \cite{KaaLes:SFU} ensure 
that the triple 
\[
\left( \Cl_{s+2,r}, \, L^2(\R, H) \ox \Lambda^* \R^2, \, A(\cdot) \ox f_1 + \partial_x \ox f_2  \right)
\]
is an unbounded representative of the product $[A(\cdot)] \ox_{C_0(\R)}[\partial_x]$. 

We now relate our product spectral triple to a $\KO$-theory class. We first 
identify $\Lambda^* \R^2 \simeq \R^4$ and take the following 
generators of a $Cl_{2,2}$--representation:
\[
 \left\{  
  \begin{pmatrix} & & 1 & \\ & & & 1 \\ 1 & & & \\ & 1 & & \end{pmatrix}, \, 
  \begin{pmatrix} & & & 1 \\ & & -1 & \\ & -1 & & \\ 1 & & & \end{pmatrix}, \, 
   \begin{pmatrix} & & -1 & \\ & & & 1 \\ 1 & & & \\ & -1 & & \end{pmatrix}, \, 
  \begin{pmatrix} & & & -1 \\ & & -1 & \\ & 1 & & \\ 1 & & & \end{pmatrix} \right\}.
\]
The product operator $T = A(\cdot) \ox f_1 + \partial_x \ox f_2$ therefore has the form 
\[
  T = \begin{pmatrix} 0_2  &  T_+^* \\ T_+ & 0_2  \end{pmatrix}, \qquad 
  T_+ =\begin{pmatrix} A(\cdot) & \partial_x\\ \partial_x & -A(\cdot) \end{pmatrix}.
\]
We see that $T_+$ is a skew-adjoint unbounded Fredholm operator on 
$L^2(\R, H)^{\oplus 2}$ anti-commuting with an \emph{ungraded} 
left $\Cl_{r,s+1}$--action with generating elements  
$\{ E_1 \otimes K_1, \ldots, E_r\otimes K_1, F_1 \otimes K_1, \ldots, F_s \otimes K_1, I \otimes L_1 \}$. 
We can now apply the isomorphism $\phi:\,\KKO(\Cl_{s+2,r},\R)\to \KO_{s+2-r}(\R)$ which 
is given by  
\[
 \phi \left( \big[(\Cl_{s+2, r}, \, L^2(\R, H)^{\oplus 4}, \, T )\big] \right) 
 = \ind_{r,s+2}(T_+) \in A_{r,s+2} \simeq KO_{s+2-r}(\R).
\]
Therefore, we find that
\[
\phi([A(\cdot)]\ox_{C_0(\R)}[\partial_x]) 
 = \ind_{r, s+2} \begin{pmatrix} A(\cdot) & \partial_x\\ \partial_x & -A(\cdot) \end{pmatrix} 
 = - \sff_{r,s+2}(A(\cdot)).
 \]
 The last equality comes from Theorem~\ref{thm:Robbin-Salamon} and 
 by conjugating $T_+$ by the self-adjoint unitary $\frac 1{\sqrt 2}(K_1-K_2)$, 
 cf. Remark~\ref{rk:RS_Clifford_normalization}, which will reverse the 
 orientation of $\Cl_{2,0}$ and sends the class in $KO_{s+2-r}(\R)$ to its inverse.
\end{proof}

The minus sign relating the Kasparov product to the $KO$--valued 
spectral flow is common in such index formulas. In the case that 
$r=s$ and $r-s = 1$, $SF_{r,s+2}$ has range $\Z/2\Z$ and so this 
minus sign can be ignored.

\appendix

%*********************************************************
\section{Some classical homotopy equivalences}
\label{s.SCHE}
%*********************************************************

In this appendix we will explicitly describe a few useful homotopy
equivalences most of which are reformulations of results due to Atiyah
and Singer \cite{AS69}. We will use freely the notation introduced in
Sec.~\ref{ss.PN}.

Strictly speaking, \cite{AS69} cover ``only'' the case of $\Cl_{0,k}$
symmetry. Note, however, that due to the isomorphisms
\Eqref{eq.ABSC.1}--\Eqref{eq.ABSC.4} one has the following: if
$X^{r,s}$ denotes any of the spaces $\sF^{r,s+1},
\tO_{r,s},\ldots $ etc. then $X^{r+1,s+1}$ is naturally
homeomorphic to $X^{r,s}$. Also, in a standard Hilbert space,
$X^{r+8,s}, X^{r,s+8}$ are homeomorphic to $X^{r,s}$ and the
homeomorphism is compatible with all the homotopy equivalences stated
below.\footnote{This is a metastatement whose details are hard
to formulate such that a lawyer will be happy; however every literate
reader will be able to fill in the correct details.
}                             %END OF FOOTNOTE
Therefore, we just summarize the results. The easy (but tedious) details are left to the reader.

\subsection{Homotopy results for standard $\Clrs$--Hilbert spaces}

\begin{prop}\label{p.SCHE.1} For a $\Cl_{r,s}$--Hilbert space $H$ the
inclusion $\bO_{r,s}\subset \tO_{r,s}$ is a weak homotopy equivalence.\footnote{A weak or singular
homotopy equivalence is a map which induces isomorphisms on all
homotopy groups}
\end{prop}
\begin{proof} As in \cite{AS69} one concludes that
  $\bO_{r,s}\hookrightarrow \tO_{r,s}\to \bigsetdef{x\in\sJ^{r,s}(\sQ(H))}{
\|x-\pi(F_s)\|_{\sQ}<2} $ is a fibration. In view of Prop.
\ref{p.NHCT.2} the base is contractible, whence the claim.
\end{proof}

Next we recall the main results of \cite{AS69}:

\begin{theorem}[\cite{AS69}, Theorems A(k), B(k)] \label{p.SCHE.2}
Let $H$ be a standard $\Cl_{r,s}$--Hilbert space.  Then
\[
  \ind_{r,s+1}:\sFs^{r,s+1} \to A_{r,s+1}=\KO_{s+1-r}(\R),\quad
    T\mapsto [ \ker T ]_{\sM_{r,s}/\sM_{r,s+1} }
\]
labels $\pi_0(\sFs^{r,s+1})$.
\end{theorem}

This means that each $T\in\sFs^{r,s+1}$ can be path connected to an
operator $T_V$ of the \emph{normal form} \Eqref{eq.ABSC.9}.

\begin{theorem}[\cite{AS69}, Theorems A(k), B(k)]\label{p.SCHE.3} Let $s>0$ or $r=s=0$.
Then for a standard $\Cl_{r,s}$--Hilbert space, the map
\[\begin{split}
  &\Fs^{r,s+1} \to \Omega( \FO_*^{r,s} ),\quad
     T\mapsto \ga_{\bullet}(T),\\
  &\ga_t(T):=\cos(\pi t) F_{s} + \sin(\pi t) T,\quad 0\le t\le 1
\end{split}
\]
is a homotopy equivalence.  The RHS is the space of paths from
$F_{s}$ to $-F_{s}$ in $\FO_*^{r,s}$.
  
In the corner case $r=s=0$ this must be read as follows: $\FO_*^{0,0}$ is
the space of essentially unitary Fredholm operators and $F_0=I$.
\end{theorem}

\begin{theorem}\label{p.SCHE.4} Let $s>0$ or $r=s=0$ and let
$H$ be a standard $\Cl_{r,s}$--Hilbert space. Then the map
\[
  \Phi_{r,s+1}: \Fs^{r,s+1} \to \ovl{\Omega}_{r,s},\quad
    T\mapsto -F_{s} e^{\pi T F_{s} }
\]
is a homotopy equivalence. Again, if $r=s=0$ then $F_0=I$.
Putting
\begin{equation}\label{eq.SCHE.2}
     \Psi: i\R \to i\R, \quad x\mapsto \begin{cases} x,& |x|\le 1,\\
       x/|x|,& |x|\ge 1,
     \end{cases}
\end{equation}
one obtains a homotopy equivalence
\[
  \Phi_\Psi: \sFs^{r,s+1} \to \tO_{r,s},\quad 
           T \mapsto -F_{s} e^{\pi \Psi(T F_{s} ) }.
\]
\end{theorem}
\begin{proof}
The claim about the first homotopy equivalence follows
from \cite{AS69}, Prop.~4.2 ($s>0$) and Prop.~3.3 ($r=s=0$).
The second claim is then a simple consequence since 
  on the deformation retract $\FO_*^{r,s+1}\subset \sFs^{r,s+1}$ 
the map $\Phi_\Psi$ coincides with the map $\Phi_{r,s+1}$.
\end{proof}

We note that $\Phi_\Psi(\pm F_{s+1})=\Phi_{r,s+1}(\pm F_{s+1}) =
F_{s}$ so that 
$\Phi_\Psi, \Phi_{r,s+1}$ sends paths from $F_{s+1}$ to $-F_{s+1}$ in
$\sFs^{r,s+1}, \, \Fs^{r,s+1}$ to \emph{loops} in $\ovl{\Omega}_{r,s}$
with base point $F_{s}$. As a corollary we obtain the characterization
of the connected components of $\tO^{r,s+1}$:

\begin{cor}\label{p.SCHE.5} Let $s>0$ or $r=s=0$ and
let $H$ be a standard $\Cl_{r,s}$--Hilbert space. Then we have
\emph{canonical} identifications 
$\pi_0(\tO_{r,s}) \simeq
  \pi_0(\bO_{r,s})\simeq A_{r,s+1}\simeq 
  \KO_{s+1-r}(\R)$.  
More concretely, each path component contains an element which,
w.r.t. a $\Cl_{r,s}$--linear decomposition $H=H_0\oplus V$,
takes the form
${F_{s}}\restr{H_0} \oplus -{F_s}\restr{V}$. Under the claimed
isomorphism this element is then mapped to
$[V]\in\sM_{r,s}/\sM_{r,s+1}$.
\end{cor}
\begin{proof} The first isomorphism is Prop.~\ref{p.SCHE.1}.
According to the remark after Theorem~\ref{p.SCHE.2} each class in
$\pi_0(\sF_*^{r,s+1})$ has a representative $T= F_{s+1} \oplus 0$ w.r.t.
the $\Cl_{r,s}$--linear decomposition $H=H_0\oplus V$ such that $H_0$
carries a $\Cl_{r,s+1}$ structure; then $\ind_{r,s+1}(T)=[V]$.
Under the homotopy equivalence $\Phi_{r,s+1}$ of Theorem~\ref{p.SCHE.4}
this is mapped to
\[
   \Phi_{r,s+1}(T)  = - F_{s+1} e^{\pi F_{s+1}F_s}\restr{H_0} \oplus
      - {F_s}\restr{ V } 
      =  {F_{s}}\restr{H_0} \oplus- {F_s}\restr{ V },
\]
since $(F_{s+1}F_s)^2 = -I$.      
\end{proof}

Now we are in the position to prove the main result of this section:

\begin{theorem}\label{p.SCHE.6} Let $s>0$ or $r=s=0$ and let $H$ be a
standard $\Cl_{r,s+1}$--Hilbert space.  Then the fundamental group
$\pi_1(\tilde\Omega_{r,s}, F_s)$ is canonically isomorphic to
$\sM_{r,s+1}/\sM_{r,s+2}\simeq \KO_{s+2-r}(\R)$.

The isomorphism is given as follows: given a loop 
$\gamma:[0,1]\to \tilde\Omega_{r,s}, \gamma(0)=\gamma(1)=F_s$,
then there is a $\Cl_{r,s+1}$ decomposition $H=H_0\oplus V$
with $\dim V<\infty$ and $H_0$ carrying a $\Cl_{r,s+2}$-structure
such that $\gamma$ is homotopic to the loop
\begin{align*}
\tilde\gamma(t) 
   &:= {F_s}\restr{H_0} \oplus F_s e^{-2\pi t F_{s+1}F_s} \restr V\\
   &= {F_s}\restr{ H_0 } \oplus \Bigl( \cos(2\pi t) F_s - \sin(2\pi t)
                F_{s+1} \Bigr)\restr V,\quad 0\le t \le 1.
\end{align*}
The element $[\gamma]$ is then mapped to $[V]\in\sM_{r,s+1}/\sM_{r,s+2}$. 

As usual, in the case $r=s=0$ one has to read the formulas with $F_0=I$.
\end{theorem}
\begin{proof} In view of Prop.~\ref{p.SCHE.1} it suffices to prove the
  result with $\ovl{\Omega}_{r,s}$ instead of $\tilde\Omega_{r,s}$.

Pick a typical representative $T$ of $\sM_{r,s+1}/\sM_{r,s+2}$ as in the
proof of Cor.~\ref{p.SCHE.5}. According to the isomorphism in
Theorem~\ref{p.SCHE.2} this labels $\pi_0(\FO_*^{r,s+2})$. Under the
isomorphism in Theorem~\ref{p.SCHE.3},  $\pi_0(\FO_*^{r,s+2})$ is isomorphic
to the homotopy classes of paths in $\FO_*^{r,s+1}$ from
$F_{s+1}$ to $-F_{s+1}$. For $T=F_{s+2}\oplus 0, H= H_0 \oplus V$ we
first compute $\ga_t(T)$ of Theorem~\ref{p.SCHE.3}: on $H_0$ we have
$\ga_t( F_{s+2}) = \cos(\pi t) F_{s+1}+\sin(\pi t) F_{s+2}$ while on
$V$ we have $\ga_t(0) = \cos(\pi t) F_{s+1}$.

Finally, the map $\Phi_{r,s+1}$ of Theorem~\ref{p.SCHE.4} maps the
homotopy classes of paths in $\FO_*^{r,s+1}$ from $F_{s+1}$ to $-F_{s+1}$
to $\pi_1(\ovl\Omega_{r,s},F_s)$.  On $H_0$ we have 
$\Phi_{r,s+1}\bl \ga_t(F_{s+2})\br 
   = -F_s e^{\pi \ga_t(F_{s+2}F_s)}\equiv F_s$ since $\ga_t(F_{s+1})$
anti-commutes with $F_s$, is skew and has square $-1$. 

On $V$ we have
$\Phi_{r,s+1}(\ga_t(0)) = - F_s e^{\pi \cos(\pi t) F_{s+1} F_s}, 0\le t\le 1$. 
As $\cos(\pi t)$ runs from $1$ to $-1$ this path is homotopic to 
$-F_s e^{\pi (1-2t) F_{s+1}F_s} 
= F_s e^{-2\pi t F_{s+1}F_s}, 0\le t \le 1$, completing the proof.
\end{proof}

%*********************************************************************************
\subsection{Stable homotopy in the finite dimensional case}
\label{ss.SHFC}
%*********************************************************************************

Up to this point the results of this Appendix~\ref{s.SCHE} are proved
by \emph{finite dimensional} approximation. This means that they do
have finite dimensional analogues. Here, the infinite multiplicity
assumption built-in to the notion of a standard $\Cl_{r,s}$--Hilbert space
needs to be replaced by looking at \emph{stable homotopy} groups.

We will need the stable analogues of Cor.~\ref{p.SCHE.5} and
Theorem~\ref{p.SCHE.6}. The statements can be extracted from
\cite{AS69}. However, they can also be found in Milnor's book
\cite[Sec.~24]{Mil63}.

For a finite dimensional $\Cl_{r,s}$--module $V$ one has
$\ovl\Omega_{r,s}(V)=\tO_{r,s}(V)=:\Omega_{r,s}(V)$. If $V\hookrightarrow W$
is an inclusion of finite dimensional $\Clrs$--modules then
there is an obvious inclusion $\Omega_{r,s}(V)\to \Omega_{r,s}(W)$
by sending $J\mapsto J\oplus F_s\restr{V^\perp}$. Maps
$f,g:X\to\Omega_{r,s}(V)$ are said to be \emph{stably} homotopic
if they become homotopic after embedding into $\Omega_{r,s}(W)$ for
a sufficiently high dimensional $\Clrs$--module $W\supset V$.
In this way one obtains stable homotopy sets/groups 
$\pi_0\stab(\Omega_{r,s}(V))$,  $\pi_1\stab(\Omega_{r,s}(V))$.

\begin{theorem} Let $H$ be a finite dimensional $\Clrs$--module.
Then the stable homotopy set $\pi_0\stab(\Omega_{r,s}(V))$ is
  canonically isomorphic to $\sM_{r,s}/\sM_{r,s+1}$. The concrete
generators are the same as those mentioned in Cor.
\textup{\ref{p.SCHE.5}}.

Similarly, for a finite dimensional $\Cl_{r,s+1}$--module
$H$, the stable fundamental group
  $\pi_1\stab(\Omega_{r,s}(H),F_s)$ is canonically isomorphic to
$\sM_{r,s+1}/\sM_{r,s+2}$. The concrete generators are the
same as those mentioned in Theorem \textup{\ref{p.SCHE.6}}.
\end{theorem}

%%%%%%%%%%%%%%%%%%%%%%%%%
% BIB
%\bibliography{localbib}
%\bibliographystyle{amsalpha-lmp}
% 11.07.19 ML main.bbl copied to bibliography.tex
% after editing of bibliography do not forget to save
% a patch of the diffs from main.bbl to the new bibliography.tex
\providecommand{\bysame}{\leavevmode\hbox to3em{\hrulefill}\thinspace}
\providecommand{\MR}{\relax\ifhmode\unskip\space\fi MR }
\providecommand{\MRhref}[2]{%
  \href{http://www.ams.org/mathscinet-getitem?mr=#1}{#2}
}
\providecommand{\href}[2]{#2}

%%%%%%%%%%%%%%%%%%%%%%%%%

\end{document}